\documentclass[12pt,high,twoside]{amsart}
\usepackage{amsmath,amsthm,amsfonts,amssymb,mathrsfs}
\usepackage{color}

\usepackage{amssymb,cite}
\usepackage[colorlinks,plainpages,citecolor=magenta, linkcolor=blue, backref]{hyperref}

\usepackage{eucal}

\usepackage{tikz}

\usepackage{hyperref}
\date{\today}

\usepackage{hyperref}
\input{xy}
 \xyoption{all}
 \xyoption{arc}

\newtheorem{theorem}{Theorem}

\newtheorem{proposition}{Proposition}
\newtheorem{corollary}{Corollary}

\newtheorem{lemma}{Lemma}
\theoremstyle{definition}
\newtheorem{definition}{Definition}
\newtheorem{example}{Example}
\newtheorem{remark}{Remark}

\begin{document}

\title[On the semigroup $\boldsymbol{B}_{\omega}^{\mathscr{F}_n}$ which is generated by the family $\mathscr{F}_n$]{On the semigroup $\boldsymbol{B}_{\omega}^{\mathscr{F}_n}$ which is generated by the family $\mathscr{F}_n$ of finite bounded intervals of $\omega$}
\author{Oleg Gutik and Olha Popadiuk}
\address{Faculty of Mechanics and Mathematics,
Lviv University, Universytetska 1, Lviv, 79000, Ukraine}
\email{oleg.gutik@lnu.edu.ua, o.popadiuk@gmail.com}



\begin{abstract}
We study the semigroup $\boldsymbol{B}_{\omega}^{\mathscr{F}}$, which is introduced in the paper [O. Gutik and M. Mykhalenych, \emph{On some generalization of the bicyclic monoid}, Visnyk Lviv. Univ. Ser. Mech.-Mat.
\textbf{90} (2020), 5--19 (in Ukrainian)], in the case when the family $\mathscr{F}_n$ generated by the set $\{0,1,\ldots,n\}$. We show that the Green relations $\mathscr{D}$ and $\mathscr{J}$ coincide in $\boldsymbol{B}_{\omega}^{\mathscr{F}_n}$, the semigroup $\boldsymbol{B}_{\omega}^{\mathscr{F}_n}$ is isomorphic to the semigroup $\mathscr{I}_\omega^{n+1}(\overrightarrow{\mathrm{conv}})$ of partial convex order isomorphisms of $(\omega,\leqslant)$ of the rank $\leqslant n+1$, and  $\boldsymbol{B}_{\omega}^{\mathscr{F}_n}$ admits only Rees congruences. Also, we study shift-continuous topologies on the semigroup $\boldsymbol{B}_{\omega}^{\mathscr{F}_n}$. In particular we prove that for any shift-continuous $T_1$-topology $\tau$ on the semigroup $\boldsymbol{B}_{\omega}^{\mathscr{F}_n}$ every non-zero element of $\boldsymbol{B}_{\omega}^{\mathscr{F}_n}$ is an isolated point of $(\boldsymbol{B}_{\omega}^{\mathscr{F}_n},\tau)$, $\boldsymbol{B}_{\omega}^{\mathscr{F}_n}$ admits the unique compact shift-continuous $T_1$-topology, and every $\omega_{\mathfrak{d}}$-compact shift-continuous $T_1$-topology is compact. We describe the closure of the semigroup $\boldsymbol{B}_{\omega}^{\mathscr{F}_n}$ in a Hausdorff semitopological semigroup and prove the criterium when a topological inverse semigroup $\boldsymbol{B}_{\omega}^{\mathscr{F}_n}$ is $H$-closed in the class of Hausdorff topological semigroups.
\end{abstract}

\maketitle


\section{Introduction, motivation and main definitions}

We shall follow the terminology of~\cite{Carruth-Hildebrant-Koch-1983, Clifford-Preston-1961, Clifford-Preston-1967, Engelking-1989, Ruppert-1984}. By $\omega$ we denote the set of all non-negative integers.

Let $\mathscr{P}(\omega)$ be  the family of all subsets of $\omega$. For any $F\in\mathscr{P}(\omega)$ and $n,m\in\omega$ we put $n-m+F=\{n-m+k\colon k\in F\}$ if $F\neq\varnothing$ and $n-m+\varnothing=\varnothing$. A subfamily $\mathscr{F}\subseteq\mathscr{P}(\omega)$ is called \emph{${\omega}$-closed} if $F_1\cap(-n+F_2)\in\mathscr{F}$ for all $n\in\omega$ and $F_1,F_2\in\mathscr{F}$.

We denote $[0;0]=\{0\}$ and $[0;k]=\{0,\ldots,k\}$ for any positive integer $k$. The set $[0;k]$, $k\in\omega$, is called an \emph{initial interval} of $\omega$.

A \emph{partially ordered set} (or shortly a \emph{poset}) $(X,\leqq)$ is the set $X$ with the reflexive, antisymmetric and transitive relation $\leqq$. In this case relation $\leqq$ is called a partial order on $X$. A partially ordered set $(X,\leqq)$ is \emph{linearly ordered} or is a \emph{chain} if $x\leqq y$ or $y\leqq x$ for any $x,y\in X$. A map $f$ from a poset $(X,\leqq)$ onto a poset $(Y,\eqslantless)$ is said to be an order isomorphism if $f$ is bijective and $x\leqq y$ if and only if $f(x)\eqslantless f(y)$. A \emph{partial order isomorphism} $f$ from a poset $(X,\leqq)$ into a poset $(Y,\eqslantless)$ is an order isomorphism from a subset $A$ of a poset $(X,\leqq)$ into a subset $B$ of a poset $(Y,\eqslantless)$. For any elements $x$ of a poset $(X,\leqq)$ we denote
\begin{equation*}
  {\uparrow_{\leqq}}x=\{y\in X\colon x\leqq y\} \qquad \hbox{and} \qquad {\downarrow_{\leqq}}x=\{y\in X\colon y\leqq x\}.
\end{equation*}

A semigroup $S$ is called {\it inverse} if for any
element $x\in S$ there exists a unique $x^{-1}\in S$ such that
$xx^{-1}x=x$ and $x^{-1}xx^{-1}=x^{-1}$. The element $x^{-1}$ is
called the {\it inverse of} $x\in S$. If $S$ is an inverse
semigroup, then the mapping $\operatorname{inv}\colon S\to S$
which assigns to every element $x$ of $S$ its inverse element
$x^{-1}$ is called the {\it inversion}.

If $S$ is a semigroup, then we shall denote the subset of all
idempotents in $S$ by $E(S)$. If $S$ is an inverse semigroup, then
$E(S)$ is closed under multiplication and we shall refer to $E(S)$ as a
\emph{band} (or the \emph{band of} $S$). Then the semigroup
operation on $S$ determines the following partial order $\preccurlyeq$
on $E(S)$: $e\preccurlyeq f$ if and only if $ef=fe=e$. This order is
called the {\em natural partial order} on $E(S)$. A \emph{semilattice} is a commutative semigroup of idempotents. By $(\omega,\min)$ or $\omega_{\min}$ we denote the set $\omega$ with the semilattice operation $x\cdot y=\min\{x,y\}$.

If $S$ is an inverse semigroup then the semigroup operation on $S$ determines the following partial order $\preccurlyeq$
on $S$: $s\preccurlyeq t$ if and only if there exists $e\in E(S)$ such that $s=te$. This order is
called the {\em natural partial order} on $S$ \cite{Wagner-1952}.

For semigroups $S$ and $T$ a map $\mathfrak{h}\colon S\to T$ is called a \emph{homomorphism} if $\mathfrak{h}(s_1\cdot s_2)=\mathfrak{h}(s_1)\cdot \mathfrak{h}(s_2)$ for all $s_1,s_2\in S$.

A \emph{congruence} on a semigroup $S$ is an equivalence relation $\mathfrak{C}$ on $S$ such that $(s,t)\in\mathfrak{C}$ implies that $(as,at),(sb,tb)\in\mathfrak{C}$ for all $a,b\in S$. Every congruence $\mathfrak{C}$ on a semigroup $S$ generates the \emph{associated natural homomorphism} $\mathfrak{C}^\natural\colon S\to S/\mathfrak{C}$ which assigns to each element $s$ of $S$ its congruence class $[s]_\mathfrak{C}$ in the quotient semigroup $S/\mathfrak{C}$. Also every homomorphism $\mathfrak{h}\colon S\to T$ of semigroups $S$ and $T$ generates the congruence $\mathfrak{C}_\mathfrak{h}$ on $S$: $(s_1,s_2)\in\mathfrak{C}_\mathfrak{h}$ if and only if $\mathfrak{h}(s_1)=\mathfrak{h}(s_2)$.

A nonempty subset $I$ of a semigroup $S$ is called an \emph{ideal} of $S$ if $SIS=\{asb\colon s\in I, \; a,b\in S\}\subseteq I$. Every ideal $I$ of a semigroup $S$ generates the congruence $\mathfrak{C}_I=(I\times I)\cup\Delta_S$ on $S$, which is called the \emph{Rees congruence} on $S$.

Let $\mathscr{I}_\lambda$ denote the set of all partial one-to-one transformations of $\lambda$ together with the following semigroup operation:
\begin{equation*}
    x(\alpha\beta)=(x\alpha)\beta \quad \mbox{if} \quad
    x\in\operatorname{dom}(\alpha\beta)=\{
    y\in\operatorname{dom}\alpha\colon
    y\alpha\in\operatorname{dom}\beta\}, \qquad \mbox{for} \quad
    \alpha,\beta\in\mathscr{I}_\lambda.
\end{equation*}
The semigroup $\mathscr{I}_\lambda$ is called the \emph{symmetric
inverse semigroup} over the cardinal $\lambda$~(see \cite{Clifford-Preston-1961}). For any $\alpha\in\mathscr{I}_\lambda$ the cardinality of $\operatorname{dom}\alpha$ is called the \emph{rank} of $\alpha$ and it is denoted by $\operatorname{rank}\alpha$. The symmetric inverse semigroup was introduced by V.~V.~Wagner~\cite{Wagner-1952}
and it plays a major role in the theory of semigroups.


Put
$\mathscr{I}_\lambda^n=\{ \alpha\in\mathscr{I}_\lambda\colon
\operatorname{rank}\alpha\leqslant n\}$,
for $n=1,2,3,\ldots$. Obviously,
$\mathscr{I}_\lambda^n$ ($n=1,2,3,\ldots$) are inverse semigroups,
$\mathscr{I}_\lambda^n$ is an ideal of $\mathscr{I}_\lambda$, for each $n=1,2,3,\ldots$. The semigroup
$\mathscr{I}_\lambda^n$ is called the \emph{symmetric inverse semigroup of
finite transformations of the rank $\leqslant n$} \cite{Gutik-Reiter-2009}. By
\begin{equation*}
\left({%
\begin{smallmatrix}
  x_1 & x_2 & \cdots & x_n \\
  y_1 & y_2 & \cdots & y_n \\
\end{smallmatrix}%
}\right)
\end{equation*}
we denote a partial one-to-one transformation which maps $x_1$ onto $y_1$, $x_2$ onto $y_2$, $\ldots$, and $x_n$ onto $y_n$. Obviously, in such case we have $x_i\neq x_j$ and $y_i\neq y_j$ for $i\neq j$ ($i,j=1,2,3,\ldots,n$). The empty partial map $\varnothing\colon \lambda\rightharpoonup\lambda$ is denoted by $\boldsymbol{0}$. It is obvious that $\boldsymbol{0}$ is zero of the semigroup $\mathscr{I}_\lambda^n$.

For a partially ordered set $(P, \leqq)$, a subset $X$ of $P$ is called \emph{order-convex}, if $x\leqq z\leqq y$ and $\{x, y\}\subset X$ implies that $z\in X$, for all $x, y, z\in P$ \cite{Harzheim=2005}. It is obvious that the set of all partial order isomorphisms between convex subsets of $(\omega,\leqslant)$ under the composition of partial self-maps forms an inverse subsemigroup of the symmetric inverse semigroup $\mathscr{I}_\omega$ over the set $\omega$. We denote this semigroup by $\mathscr{I}_\omega(\overrightarrow{\mathrm{conv}})$. We put $\mathscr{I}_\omega^n(\overrightarrow{\mathrm{conv}})=\mathscr{I}_\omega(\overrightarrow{\mathrm{conv}})\cap\mathscr{I}_\omega^n$ and it is obvious that $\mathscr{I}_\omega^n(\overrightarrow{\mathrm{conv}})$ is closed under the semigroup operation of $\mathscr{I}_\omega^n$ and the semigroup $\mathscr{I}_\omega^n(\overrightarrow{\mathrm{conv}})$ is called the \emph{inverse semigroup of
convex order isomorphisms of $(\omega,\leqslant)$ of the rank $\leqslant n$}.

The bicyclic monoid ${\mathscr{C}}(p,q)$ is the semigroup with the identity $1$ generated by two elements $p$ and $q$ subjected only to the condition $pq=1$. The semigroup operation on ${\mathscr{C}}(p,q)$ is determined as follows:
\begin{equation*}
    q^kp^l\cdot q^mp^n=q^{k+m-\min\{l,m\}}p^{l+n-\min\{l,m\}}.
\end{equation*}
It is well known that the bicyclic monoid ${\mathscr{C}}(p,q)$ is a bisimple (and hence simple) combinatorial $E$-unitary inverse semigroup and every non-trivial congruence on ${\mathscr{C}}(p,q)$ is a group congruence \cite{Clifford-Preston-1961}.

On the set $\boldsymbol{B}_{\omega}=\omega\times\omega$ we define the semigroup operation ``$\cdot$'' in the following way
\begin{equation*}
  (i_1,j_1)\cdot(i_2,j_2)=
  \left\{
    \begin{array}{ll}
      (i_1-j_1+i_2,j_2), & \hbox{if~} j_1\leqslant i_2;\\
      (i_1,j_1-i_2+j_2), & \hbox{if~} j_1\geqslant i_2.
    \end{array}
  \right.
\end{equation*}
It is well known that the semigroup $\boldsymbol{B}_{\omega}$ is isomorphic to the bicyclic monoid by the mapping $\mathfrak{h}\colon \mathscr{C}(p,q)\to \boldsymbol{B}_{\omega}$, $q^kp^l\mapsto (k,l)$ (see: \cite[Section~1.12]{Clifford-Preston-1961} or \cite[Exercise IV.1.11$(ii)$]{Petrich-1984}).

By $\mathbb{R}$ and $\omega_{\mathfrak{d}}$ we denote the set of real numbers with the usual topology and the infinite countable discrete space, respectively.

Let $Y$ be a topological space. A topological space $X$ is called:
\begin{itemize}
  \item \emph{compact} if any open cover of $X$ contains a finite subcover;
  \item \emph{countably compact} if each closed discrete subspace of $X$ is finite;
  \item \emph{$Y$-compact} if every continuous image of $X$ in $Y$ is compact.
\end{itemize}

A {\it topological} ({\it semitopological}) {\it semigroup} is a topological space together with a continuous (separately continuous) semigroup operation. If $S$ is a~semigroup and $\tau$ is a topology on $S$ such that $(S,\tau)$ is a topological semigroup, then we shall call $\tau$ a \emph{semigroup} \emph{topology} on $S$, and if $\tau$ is a topology on $S$ such that $(S,\tau)$ is a semitopological semigroup, then we shall call $\tau$ a \emph{shift-continuous} \emph{topology} on~$S$. An inverse topological semigroup with the continuous inversion is called a \emph{topological inverse semigroup}. 

Next we shall describe the construction which is introduced in \cite{Gutik-Mykhalenych=2020}.

Let $\boldsymbol{B}_{\omega}$ be the bicyclic monoid and $\mathscr{F}$ be an ${\omega}$-closed subfamily of $\mathscr{P}(\omega)$. On the set $\boldsymbol{B}_{\omega}\times\mathscr{F}$ we define the semigroup operation ``$\cdot$'' in the following way
\begin{equation*}
  (i_1,j_1,F_1)\cdot(i_2,j_2,F_2)=
  \left\{
    \begin{array}{ll}
      (i_1-j_1+i_2,j_2,(j_1-i_2+F_1)\cap F_2), & \hbox{if~} j_1\leqslant i_2;\\
      (i_1,j_1-i_2+j_2,F_1\cap (i_2-j_1+F_2)), & \hbox{if~} j_1\geqslant i_2.
    \end{array}
  \right.
\end{equation*}
In \cite{Gutik-Mykhalenych=2020} it is proved that if the family $\mathscr{F}\subseteq\mathscr{P}(\omega)$ is ${\omega}$-closed then $(\boldsymbol{B}_{\omega}\times\mathscr{F},\cdot)$ is a semigroup. Moreover, if an ${\omega}$-closed family  $\mathscr{F}\subseteq\mathscr{P}(\omega)$ contains the empty set $\varnothing$ then the set
$ 
  \boldsymbol{I}=\{(i,j,\varnothing)\colon i,j\in\omega\}
$ 
is an ideal of the semigroup $(\boldsymbol{B}_{\omega}\times\mathscr{F},\cdot)$. For any ${\omega}$-closed family $\mathscr{F}\subseteq\mathscr{P}(\omega)$ the following semigroup
\begin{equation*}
  \boldsymbol{B}_{\omega}^{\mathscr{F}}=
\left\{
  \begin{array}{ll}
    (\boldsymbol{B}_{\omega}\times\mathscr{F},\cdot)/\boldsymbol{I}, & \hbox{if~} \varnothing\in\mathscr{F};\\
    (\boldsymbol{B}_{\omega}\times\mathscr{F},\cdot), & \hbox{if~} \varnothing\notin\mathscr{F}
  \end{array}
\right.
\end{equation*}
is defined in \cite{Gutik-Mykhalenych=2020}. The semigroup $\boldsymbol{B}_{\omega}^{\mathscr{F}}$ generalizes the bicyclic monoid and the countable semigroup of matrix units. It is proved in \cite{Gutik-Mykhalenych=2020} that $\boldsymbol{B}_{\omega}^{\mathscr{F}}$ is a combinatorial inverse semigroup and Green's relations, the natural partial order on $\boldsymbol{B}_{\omega}^{\mathscr{F}}$ and its set of idempotents are described. The criteria of simplicity, $0$-simplicity, bisimplicity, $0$-bisimplicity of the semigroup $\boldsymbol{B}_{\omega}^{\mathscr{F}}$ and when $\boldsymbol{B}_{\omega}^{\mathscr{F}}$ has the identity, is isomorphic to the bicyclic semigroup or the countable semigroup of matrix units are given. In particularly in \cite{Gutik-Mykhalenych=2020} is proved that the semigroup $\boldsymbol{B}_{\omega}^{\mathscr{F}}$ is isomorphic to the semigrpoup of ${\omega}{\times}{\omega}$-matrix units if and only if $\mathscr{F}$ consists of a singleton set and the empty set.

The semigroup $\boldsymbol{B}_{\omega}^{\mathscr{F}}$ in the case when the family $\mathscr{F}$ consists of the empty set and some singleton subsets of $\omega$ is studied in \cite{Gutik-Lysetska=2021}. It is proved that the semigroup $\boldsymbol{B}_{\omega}^{\mathscr{F}}$ is isomorphic to the  subsemigroup $\mathscr{B}_{\omega}^{\Rsh}(\boldsymbol{F}_{\min})$ of the Brandt $\omega$-extension of the subsemilattice $(\boldsymbol{F},\min)$ of $(\omega,\min)$, where $\boldsymbol{F}=\bigcup\mathscr{F}$. Also topologizations of the semigroup  $\boldsymbol{B}_{\omega}^{\mathscr{F}}$ and its closure in semitopological semigroups are studied.

For any $n\in\omega$ we put $\mathscr{F}_n=\left\{[0;k]\colon k=0,\ldots,n\right\}$. It is obvious that $\mathscr{F}_n$ is an $\omega$-closed family of $\omega$.

In this paper we study the semigroup $\boldsymbol{B}_{\omega}^{\mathscr{F}_n}$. We show that the Green relations $\mathscr{D}$ and $\mathscr{J}$ coincide in $\boldsymbol{B}_{\omega}^{\mathscr{F}_n}$, the semigroup $\boldsymbol{B}_{\omega}^{\mathscr{F}_n}$ is isomorphic to the semigroup $\mathscr{I}_\omega^{n+1}(\overrightarrow{\mathrm{conv}})$, and  $\boldsymbol{B}_{\omega}^{\mathscr{F}_n}$ admits only Rees congruences. Also, we study shift-continuous topologizations of the semigroup $\boldsymbol{B}_{\omega}^{\mathscr{F}_n}$. In particular we prove that for any shift-continuous $T_1$-topology $\tau$ on the semigroup $\boldsymbol{B}_{\omega}^{\mathscr{F}_n}$ every non-zero element of $\boldsymbol{B}_{\omega}^{\mathscr{F}_n}$ is an isolated point of $(\boldsymbol{B}_{\omega}^{\mathscr{F}_n},\tau)$, $\boldsymbol{B}_{\omega}^{\mathscr{F}_n}$ admits the unique compact shift-continuous $T_1$-topology, and every $\omega_{\mathfrak{d}}$-compact shift-continuous $T_1$-topology is compact. We describe the closure of the semigroup $\boldsymbol{B}_{\omega}^{\mathscr{F}_n}$ in a Hausdorff semitopological semigroup and prove the criterium when a topological inverse semigroup $\boldsymbol{B}_{\omega}^{\mathscr{F}_n}$ is $H$-closed in the class of Hausdorff topological semigroups.

\section{Algebraic properties of the semigroup $\boldsymbol{B}_{\omega}^{\mathscr{F}_n}$}

An inverse semigroup $S$ with zero is said to be $0$-$E$-\emph{unitary} if $0\neq e\preccurlyeq s$, where $e$ is an idempotent in $S$, implies that $s$ is an idempotent \cite{Lawson-1998}. The class of $0$-$E$-unitary semigroups was first defined by Maria Szendrei \cite{Szendrei-1987},
although she called them $E^*$-unitary. The term $0$-$E$-unitary appears to be due
to Meakin and Sapir \cite{Meakin-Sapir-1993}.

In the following proposition we summarise properties which follow from properties of the semigroup $\boldsymbol{B}_{\omega}^{\mathscr{F}}$ in the general case. These properties are corollaries of the results of the paper \cite{Gutik-Mykhalenych=2020}.

\begin{proposition}\label{proposition-2.1}
For any $n\in\omega$  the following statements hold:
\begin{enumerate}
  \item\label{proposition-2.1(1)} $\boldsymbol{B}_{\omega}^{\mathscr{F}_n}$ is an inverse semigroup, namely $\boldsymbol{0}^{-1}=\boldsymbol{0}$ and $(i,j,[0;k])^{-1}=(j,i,[0;k])$, for any $i,j,k\in\omega$;
  \item\label{proposition-2.1(2)} $(i,j,[0;k])\in\boldsymbol{B}_{\omega}^{\mathscr{F}_n}$ is an idempotent if and only if $i=j$;
  \item\label{proposition-2.1(3)} $(i_1,i_1,[0;k_1])\preccurlyeq(i_2,i_2,[0;k_2])$ in $E(\boldsymbol{B}_{\omega}^{\mathscr{F}_n})$ if and only if $i_1\geqslant i_2$ and $i_1+k_1\leqslant i_2+k_2$ and this natural partial order on $E(\boldsymbol{B}_{\omega}^{\mathscr{F}_n})$ is presented on Fig.~\ref{fig-2.1};
\begin{figure}[h]
\vskip.1cm
\begin{center}
\tiny{
\begin{tikzpicture}[scale=.7]
\draw (0,-4) node {$\boldsymbol{0}$};
\draw (0,0) node {$(0,0,[0;0])$};
\draw[>=latex,->,thick] (0,-.3) -- (0,-3.7);
\draw (3,0) node {$(1,1,[0;0])$};
\draw[>=latex,->,thick] (3,-.3) -- (0.2,-3.7);
\draw (6,0) node {$(2,2,[0;0])$};
\draw[>=latex,->,thick] (6,-.3) -- (0.3,-3.75);
\draw (9,0) node {$(3,3,[0;0])$};
\draw[>=latex,->,thick] (9,-.3) -- (0.45,-3.8);
\draw (12,0) node {$(4,4,[0;0])$};
\draw[>=latex,->,thick] (12,-.3) -- (0.6,-3.85);
\draw (15,0) node {$\boldsymbol{\cdots}$};
\draw (10.,-1.5) node {$\boldsymbol{\cdots}$};
\draw (18,0) node {$(i,i,[0;0])$};
\draw[>=latex,->,thick] (18,-.3) -- (0.8,-3.9);
\draw (21,0) node {$(i{+}1,i{+}1,[0;0])$};
\draw[>=latex,->,thick] (21,-.3) -- (0.99,-3.97);
\draw (24,0) node {$\boldsymbol{\cdots}$};
\draw (18.,-1.5) node {$\boldsymbol{\cdots}$};
\draw (0,2) node {$(0,0,[0;1])$};
\draw[>=latex,->,thick] (0,1.7) -- (0,0.3);
\draw[>=latex,->,thick] (0.3,1.7) -- (2.7,0.3);
\draw (3,2) node {$(1,1,[0;1])$};
\draw[>=latex,->,thick] (3,1.7) -- (3,0.3);
\draw[>=latex,->,thick] (3.3,1.7) -- (5.7,0.3);
\draw (6,2) node {$(2,2,[0;1])$};
\draw[>=latex,->,thick] (6,1.7) -- (6,0.3);
\draw[>=latex,->,thick] (6.3,1.7) -- (8.7,0.3);
\draw (9,2) node {$(3,3,[0;1])$};
\draw[>=latex,->,thick] (9,1.7) -- (9,0.3);
\draw[>=latex,->,thick] (9.3,1.7) -- (11.7,0.3);
\draw (12,2) node {$(4,4,[0;1])$};
\draw[>=latex,->,thick] (12,1.7) -- (12,0.3);
\draw[>=latex,->,thick] (12.3,1.7) -- (13.6,0.9);
\draw (15,2) node {$\boldsymbol{\cdots}$};
\draw[>=latex,->,thick] (16.3,1.) -- (17.6,0.3);
\draw (18,2) node {$(i,i,[0;1])$};
\draw[>=latex,->,thick] (18,1.7) -- (18,0.3);
\draw[>=latex,->,thick] (18.3,1.7) -- (20.6,0.3);
\draw (21,2) node {$(i{+}1,i{+}1,[0;1])$};
\draw[>=latex,->,thick] (21,1.7) -- (21,0.3);
\draw[>=latex,->,thick] (21.3,1.7) -- (22.6,0.9);
\draw (24,2) node {$\boldsymbol{\cdots}$};
\draw (0,4) node {$(0,0,[0;2])$};
\draw[>=latex,->,thick] (0,3.7) -- (0,2.3);
\draw[>=latex,->,thick] (0.3,3.7) -- (2.7,2.3);
\draw (3,4) node {$(1,1,[0;2])$};
\draw[>=latex,->,thick] (3,3.7) -- (3,2.3);
\draw[>=latex,->,thick] (3.3,3.7) -- (5.7,2.3);
\draw (6,4) node {$(2,2,[0;2])$};
\draw[>=latex,->,thick] (6,3.7) -- (6,2.3);
\draw[>=latex,->,thick] (6.3,3.7) -- (8.7,2.3);
\draw (9,4) node {$(3,3,[0;2])$};
\draw[>=latex,->,thick] (9,3.7) -- (9,2.3);
\draw[>=latex,->,thick] (9.3,3.7) -- (11.7,2.3);
\draw (12,4) node {$(4,4,[0;2])$};
\draw[>=latex,->,thick] (12,3.7) -- (12,2.3);
\draw[>=latex,->,thick] (12.3,3.7) -- (13.6,2.9);
\draw (15,4) node {$\boldsymbol{\cdots}$};
\draw[>=latex,->,thick] (16.3,3.) -- (17.6,2.3);
\draw (18,4) node {$(i,i,[0;2])$};
\draw[>=latex,->,thick] (18,3.7) -- (18,2.3);
\draw[>=latex,->,thick] (18.3,3.7) -- (20.6,2.3);
\draw (21,4) node {$(i{+}1,i{+}1,[0;2])$};
\draw[>=latex,->,thick] (21,3.7) -- (21,2.3);
\draw[>=latex,->,thick] (21.3,3.7) -- (22.6,2.9);
\draw (24,4) node {$\boldsymbol{\cdots}$};
\draw (0,6) node {$(0,0,[0;3])$};
\draw[>=latex,->,thick] (0,5.7) -- (0,4.3);
\draw[>=latex,->,thick] (0.3,5.7) -- (2.7,4.3);
\draw (3,6) node {$(1,1,[0;3])$};
\draw[>=latex,->,thick] (3,5.7) -- (3,4.3);
\draw[>=latex,->,thick] (3.3,5.7) -- (5.7,4.3);
\draw (6,6) node {$(2,2,[0;3])$};
\draw[>=latex,->,thick] (6,5.7) -- (6,4.3);
\draw[>=latex,->,thick] (6.3,5.7) -- (8.7,4.3);
\draw (9,6) node {$(3,3,[0;3])$};
\draw[>=latex,->,thick] (9,5.7) -- (9,4.3);
\draw[>=latex,->,thick] (9.3,5.7) -- (11.7,4.3);
\draw (12,6) node {$(4,4,[0;3])$};
\draw[>=latex,->,thick] (12,5.7) -- (12,4.3);
\draw[>=latex,->,thick] (12.3,5.7) -- (13.6,4.9);
\draw (15,6) node {$\boldsymbol{\cdots}$};
\draw[>=latex,->,thick] (16.3,5.) -- (17.6,4.3);
\draw (18,6) node {$(i,i,[0;3])$};
\draw[>=latex,->,thick] (18,5.7) -- (18,4.3);
\draw[>=latex,->,thick] (18.3,5.7) -- (20.6,4.3);
\draw (21,6) node {$(i{+}1,i{+}1,[0;3])$};
\draw[>=latex,->,thick] (21,5.7) -- (21,4.3);
\draw[>=latex,->,thick] (21.3,5.7) -- (22.6,4.9);
\draw (24,6) node {$\boldsymbol{\cdots}$};
\draw (0,8) node {$\boldsymbol{\cdots}$};
\draw[>=latex,->,thick] (0,7.3) -- (0,6.3);
\draw[>=latex,->,thick] (1.,7.3) -- (2.7,6.3);
\draw (3,8) node {$\boldsymbol{\cdots}$};
\draw[>=latex,->,thick] (3,7.3) -- (3,6.3);
\draw[>=latex,->,thick] (4.,7.3) -- (5.7,6.3);
\draw (6,8) node {$\boldsymbol{\cdots}$};
\draw[>=latex,->,thick] (6,7.3) -- (6,6.3);
\draw[>=latex,->,thick] (7.,7.3) -- (8.7,6.3);
\draw (9,8) node {$\boldsymbol{\cdots}$};
\draw[>=latex,->,thick] (9,7.3) -- (9,6.3);
\draw[>=latex,->,thick] (10.,7.3) -- (11.7,6.3);
\draw (12,8) node {$\boldsymbol{\cdots}$};
\draw[>=latex,->,thick] (12,7.3) -- (12,6.3);
\draw[>=latex,->,thick] (13.,7.3) -- (13.7,6.9);
\draw (15,8) node {$\boldsymbol{\cdots}$};
\draw[>=latex,->,thick] (16.3,7.) -- (17.7,6.3);
\draw (18,8) node {$\boldsymbol{\cdots}$};
\draw[>=latex,->,thick] (18,7.3) -- (18,6.3);
\draw[>=latex,->,thick] (19.,7.3) -- (20.6,6.3);
\draw (21,8) node {$\boldsymbol{\cdots}$};
\draw[>=latex,->,thick] (21,7.3) -- (21,6.3);
\draw[>=latex,->,thick] (22.,7.3) -- (22.7,6.9);
\draw (24,8) node {$\boldsymbol{\cdots}$};
\draw (0,10) node {$(0,0,[0;n{-}1])$};
\draw[>=latex,->,thick] (0,9.7) -- (0,8.7);
\draw[>=latex,->,thick] (0.3,9.7) -- (2.0,8.7);
\draw (3,10) node {$(1,1,[0;n{-}1])$};
\draw[>=latex,->,thick] (3,9.7) -- (3,8.7);
\draw[>=latex,->,thick] (3.3,9.7) -- (5.0,8.7);
\draw (6,10) node {$(2,2,[0;n{-}1])$};
\draw[>=latex,->,thick] (6,9.7) -- (6,8.7);
\draw[>=latex,->,thick] (6.3,9.7) -- (8.0,8.7);
\draw (9,10) node {$(3,3,[0;n{-}1])$};
\draw[>=latex,->,thick] (9,9.7) -- (9,8.7);
\draw[>=latex,->,thick] (9.3,9.7) -- (11.0,8.7);
\draw (12,10) node {$(4,4,[0;n{-}1])$};
\draw[>=latex,->,thick] (12,9.7) -- (12,8.7);
\draw[>=latex,->,thick] (12.3,9.7) -- (13.6,8.9);
\draw (15,10) node {$\boldsymbol{\cdots}$};
\draw[>=latex,->,thick] (16.3,9.4) -- (17.,9);
\draw (18,10) node {$(i{,}i{,}[0;n{-}1\!])$};
\draw[>=latex,->,thick] (18,9.7) -- (18,8.7);
\draw[>=latex,->,thick] (18.3,9.7) -- (19.6,8.9);
\draw (21,10) node {$(i{+}1{,}i{+}1{,}[0;n{-}1\!])$};
\draw[>=latex,->,thick] (21,9.7) -- (21,8.7);
\draw[>=latex,->,thick] (21.3,9.7) -- (22.6,8.9);
\draw (24,10) node {$\boldsymbol{\cdots}$};
\draw (0,12) node {$(0,0,[0;n])$};
\draw[>=latex,->,thick] (0,11.7) -- (0,10.3);
\draw[>=latex,->,thick] (0.3,11.7) -- (2.7,10.3);
\draw (3,12) node {$(1,1,[0;n])$};
\draw[>=latex,->,thick] (3,11.7) -- (3,10.3);
\draw[>=latex,->,thick] (3.3,11.7) -- (5.7,10.3);
\draw (6,12) node {$(2,2,[0;n])$};
\draw[>=latex,->,thick] (6,11.7) -- (6,10.3);
\draw[>=latex,->,thick] (6.3,11.7) -- (8.7,10.3);
\draw (9,12) node {$(3,3,[0;n])$};
\draw[>=latex,->,thick] (9,11.7) -- (9,10.3);
\draw[>=latex,->,thick] (9.3,11.7) -- (11.7,10.3);
\draw (12,12) node {$(4,4,[0;n])$};
\draw[>=latex,->,thick] (12,11.7) -- (12,10.3);
\draw[>=latex,->,thick] (12.3,11.7) -- (13.6,10.9);
\draw (15,12) node {$\boldsymbol{\cdots}$};
\draw[>=latex,->,thick] (16.3,11.) -- (17.6,10.3);
\draw (18,12) node {$(i,i,[0;n])$};
\draw[>=latex,->,thick] (18,11.7) -- (18,10.3);
\draw[>=latex,->,thick] (18.3,11.7) -- (20.6,10.3);
\draw (21,12) node {$(i{+}1,i{+}1,[0;n])$};
\draw[>=latex,->,thick] (21,11.7) -- (21,10.3);
\draw[>=latex,->,thick] (21.3,11.7) -- (22.6,10.9);
\draw (24,12) node {$\boldsymbol{\cdots}$};
\end{tikzpicture}
}
\end{center}
\vskip.1cm
\caption{The natural partial order on the band $E(\boldsymbol{B}_{\omega}^{\mathscr{F}_n})$}\label{fig-2.1}
\end{figure}
  \item\label{proposition-2.1(4)} $(i,i,[0;n])$ is a maximal idempotent of $E(\boldsymbol{B}_{\omega}^{\mathscr{F}_n})$ for any $i\in\omega$;
  \item\label{proposition-2.1(5)} $(i,i,[0;0])$ is a primitive idempotent of $E(\boldsymbol{B}_{\omega}^{\mathscr{F}_n})$ for any $i\in\omega$;
  \item\label{proposition-2.1(6)} $(i_1,j_1,[0;k_1])\mathscr{R}(i_2,j_2,[0;k_2])$ in $\boldsymbol{B}_{\omega}^{\mathscr{F}_n}$ if and only if $i_1=i_2$ and $k_1=k_2$;
  \item\label{proposition-2.1(7)} $(i_1,j_1,[0;k_1])\mathscr{L}(i_2,j_2,[0;k_2])$ in $\boldsymbol{B}_{\omega}^{\mathscr{F}_n}$ if and only if $j_1=j_2$ and $k_1=k_2$;
  \item\label{proposition-2.1(8)} $(i_1,j_1,[0;k_1])\mathscr{H}(i_2,j_2,[0;k_2])$ in $\boldsymbol{B}_{\omega}^{\mathscr{F}_n}$ if and only if $i_1=i_2$, $j_1=j_2$ i $k_1=k_2$;
  \item\label{proposition-2.1(9)} $(i_1,j_1,[0;k_1])\mathscr{D}(i_2,j_2,[0;k_2])$ in $\boldsymbol{B}_{\omega}^{\mathscr{F}_n}$ if and only if $k_1=k_2$;
  \item\label{proposition-2.1(10)} $\mathscr{D}=\mathscr{J}$ in $\boldsymbol{B}_{\omega}^{\mathscr{F}_n}$;
  \item\label{proposition-2.1(11)} $(i_1,j_1,[0;k_1])\preccurlyeq(i_2,j_2,[0;k_2])$ in $\boldsymbol{B}_{\omega}^{\mathscr{F}_n}$ if and only if $i_1\geqslant i_2$, $i_1-j_1=i_2-j_2$ and $i_1+k_1\leqslant i_2+k_2$;
  \item\label{proposition-2.1(12)} $\boldsymbol{B}_{\omega}^{\mathscr{F}_n}$ is a $0$-$E$-unitary inverse semigroup.
\end{enumerate}
\end{proposition}

\begin{proof}
Statements \eqref{proposition-2.1(1)}--\eqref{proposition-2.1(5)} are trivial. Statements \eqref{proposition-2.1(6)}--\eqref{proposition-2.1(8)} follow from Proposition~3.2.11 of \cite{Lawson-1998} and corresponding statements of Theorem 2 in \cite{Gutik-Mykhalenych=2020}.

\smallskip

\eqref{proposition-2.1(9)} $(\Rightarrow)$ Suppose that $(i_1,j_1,[0;k_1])\mathscr{D}(i_2,j_2,[0;k_2])$ in $\boldsymbol{B}_{\omega}^{\mathscr{F}_n}$.
Then there exists $(i_0,j_0,[0;k_0])\in\boldsymbol{B}_{\omega}^{\mathscr{F}_n}$ such that $(i_1,j_1,[0;k_1])\mathscr{L}(i_0,j_0,[0;k_0])$ and $(i_0,j_0,[0;k_0])\mathscr{R}(i_2,j_2,[0;k_2])$. By statement \eqref{proposition-2.1(6)} we have that $i_0=i_2$ and $k_0=k_2$, and by \eqref{proposition-2.1(7)} we get that $j_0=j_1$ and $k_1=k_0$. This implies that  $k_1=k_2$.

\smallskip

$(\Leftarrow)$
Let $(i_1,j_1,[0;k])$ and $(i_2,j_2,[0;k])$ be elements of the semigroup $\boldsymbol{B}_{\omega}^{\mathscr{F}_n}$. By statements \eqref{proposition-2.1(6)} and \eqref{proposition-2.1(7)} we have that $(i_1,j_1,[0;k])\mathscr{L}(i_1,j_2,[0;k])\mathscr{R}(i_2,j_2,[0;k])$ and hence $(i_1,j_1,[0;k])\mathscr{D}((i_2,j_2,[0;k])$ in $\boldsymbol{B}_{\omega}^{\mathscr{F}_n}$.

\smallskip

\eqref{proposition-2.1(10)} It is obvious that the $\mathscr{D}$-class of the zero $\boldsymbol{0}$  coincides with $\{\boldsymbol{0}\}$. Also the $\mathscr{J}$-class of the zero $\boldsymbol{0}$  coincides with $\{\boldsymbol{0}\}$.

Fix an arbitrary non-zero element $(i_0,j_0,[0;k_0])$ of $\boldsymbol{B}_{\omega}^{\mathscr{F}_n}$. By \eqref{proposition-2.1(9)} the $\mathscr{D}$-class of $(i_0,j_0,[0;k_0])$ is the following set $\mathbf{D}=\{(i,j,[0;k_0])\colon i,j\in\omega\}$. By \eqref{proposition-2.1(3)} every two distinct idempotents of the set $\mathbf{D}$ are incomparable, and hence every idempotent of the  $\mathscr{D}$-class of $(i_0,j_0,[0;k_0])$  is minimal with the respect to the natural partial order on $\boldsymbol{B}_{\omega}^{\mathscr{F}_n}$. By Proposition 3.2.17 from \cite{Lawson-1998}  if the $\mathscr{D}$-class
$D_y$ has a minimal element then $D_y=J_y$ and hence the $\mathscr{D}$-class of $(i_0,j_0,[0;k_0])$ coincides with its $\mathscr{J}$-class. Therefore we obtain that $\mathscr{D}=\mathscr{J}$ in $\boldsymbol{B}_{\omega}^{\mathscr{F}_n}$.

\smallskip

\eqref{proposition-2.1(11)} By Proposition~2 of \cite{Gutik-Mykhalenych=2020} the inequality $(i_1,j_1,[0;k_1])\preccurlyeq(i_2,j_2,[0;k_2])$ is equivalent to the conditions
\begin{equation*}
  [0;k_1]\subseteq i_2-i_1+[0;k_2]=j_2-j_1+[0;k_2],
\end{equation*}
which are equivalent to
\begin{equation*}
  i_2-i_1=j_2-j_1\leqslant 0 \qquad \hbox{and} \qquad k_1\leqslant i_2-i_1+k_2.
\end{equation*}
It is obvious that the last conditions are equivalent to
\begin{equation*}
  i_1\geqslant i_2, \qquad i_1-j_1=i_2-j_2 \qquad \hbox{and} \qquad i_1+k_1\leqslant i_2+k_2,
\end{equation*}
which completes the proof of the statement.

\smallskip

Statement \eqref{proposition-2.1(12)} follows from \eqref{proposition-2.1(11)}.
\end{proof}

\begin{lemma}\label{lemma-2.2}
Let $n\in\omega$. Then ${\uparrow_{\preccurlyeq}}(i_0,j_0,[0;k_0])$ and ${\downarrow_{\preccurlyeq}}(i_0,j_0,[0;k_0])$ are finite subsets of the semigroup $\boldsymbol{B}_{\omega}^{\mathscr{F}_n}$ for any its non-zero element $(i_0,j_0,[0;k_0])$, $i_0,j_0\in\omega$, $k_0\in\{0,\ldots,n\}$.
\end{lemma}

\begin{proof}
By Proposition~\ref{proposition-2.1}\eqref{proposition-2.1(11)} there exist finitely many $i,j\in\omega$ and $k\in\{0,\ldots,n\}$ such that $(i,j,[0;k])\preccurlyeq (i_0,j_0,[0;k_0])$ for some $i,j\in\omega$ and hence the set ${\downarrow_{\preccurlyeq}}(i_0,j_0,[0;k_0])$ is finite.

The inequality $k\leqslant n$ and Proposition~\ref{proposition-2.1}\eqref{proposition-2.1(11)} imply that there exist finitely many $i,j\in\omega$ and $k\in\{0,\ldots,n\}$ such that $(i_0,j_0,[0;k_0])\preccurlyeq(i,j,[0;k])$, and hence the set ${\uparrow_{\preccurlyeq}}(i_0,j_0,[0;k_0])$ is finite, too.
\end{proof}

\begin{lemma}\label{lemma-2.3}
If $n\in\omega$ then for any $\alpha,\beta\in \boldsymbol{B}_{\omega}^{\mathscr{F}_n}$ the set $\alpha\cdot \boldsymbol{B}_{\omega}^{\mathscr{F}_n}\cdot\beta$ is finite.
\end{lemma}

\begin{proof}
The statement of the lemma is trivial when $\alpha=\boldsymbol{0}$ or $\beta=\boldsymbol{0}$.

Fix  arbitrary non-zero-elements $\alpha=(i_\alpha,j_\alpha,[0;k_\alpha])$ and $\beta=(i_\beta,j_\beta,[0;k_\beta])$ of $\boldsymbol{B}_{\omega}^{\mathscr{F}_n}$. If $i\geqslant j_\alpha+n+1$ or $j\geqslant  i_\beta+n+1$ then for any $k\in\{0,\ldots,n\}$ we have that
\begin{equation*}
  (i_\alpha,j_\alpha,[0;k_\alpha])\cdot (i,j,[0;k])=(i_\alpha-j_\alpha+i,j,(j_\alpha-i+[0;k_\alpha])\cap[0;k])=\boldsymbol{0}
\end{equation*}
and
\begin{equation*}
  (i,j,[0;k])\cdot(i_\beta,j_\beta,[0;k_\beta])=(i,j-i_\beta+j_\beta,[0;k]\cap (i_\beta-j+[0;k_\beta]))=\boldsymbol{0}.
\end{equation*}
Hence there exist only finitely many $(i,j,[0;k])\in \boldsymbol{B}_{\omega}^{\mathscr{F}_n}$ such that $\alpha\cdot (i,j,[0;k])\cdot\beta\neq \boldsymbol{0}$.
This implies the statement of the lemma.
\end{proof}

\begin{lemma}\label{lemma-2.4}
Let $n\in\omega$. Then for any non-zero elements $(i_1,j_1,[0;k_1])$ and $(i_2,j_2,[0;k_2])$ of $\boldsymbol{B}_{\omega}^{\mathscr{F}_n}$ the sets of solutions of the following equations
\begin{equation*}
  (i_1,j_1,[0;k_1])\cdot\chi=(i_2,j_2,[0;k_2]) \qquad \hbox{and} \qquad \chi\cdot(i_1,j_1,[0;k_1])=(i_2,j_2,[0;k_2])
\end{equation*}
in the semigroup $\boldsymbol{B}_{\omega}^{\mathscr{F}_n}$ are finite.
\end{lemma}

\begin{proof}
Suppose that $\chi$ is a solution of the equation  $(i_1,j_1,[0;k_1])\cdot\chi=(i_2,j_2,[0;k_2])$. The definition of the semigroup operation on the semigroup $\boldsymbol{B}_{\omega}^{\mathscr{F}_n}$ implies that $\chi\neq\boldsymbol{0}$ and $k_1\geqslant k_2$. Assume that $\chi=(i,j,[0;k])$ for some $i,j\in\omega$, $k=0,1,\ldots,n$. Then we have that
\begin{align*}
  (i_2,j_2,[0;k_2])&=(i_1,j_1,[0;k_1])\cdot(i,j,[0;k])= \\
   &=
   \left\{
     \begin{array}{ll}
       (i_1-j_1+i,j,(j_1-i+[0;k_1])\cap[0;k]), & \hbox{if~} j_1<i;\\
       (i_1,j,[0;k_1]\cap[0;k]),               & \hbox{if~} j_1=i;\\
       (i_1,j_1-i+j,[0;k_1]\cap(i-j_1+[0;k])), & \hbox{if~} j_1>i.
     \end{array}
   \right.
\end{align*}
We consider the following cases.
\begin{enumerate}
  \item If $j_1<i$ then $i=i_2-i_1+j_1$, $j=j_2$, $k\geqslant k_2$ and
 \begin{equation*}
 j_1-i+k_1=j_1-i_2+i_1-j_1+k_1=i_1-i_2+k_1\geqslant k.
 \end{equation*}
  \item If $j_1=i$ then $j=j_2$ and $k\geqslant k_2$.
  \item If $j_1>i$ then $i=i_2$, $j=j_2-j_1+i=j_2-j_1+i_2$ and $i-j_1+k=i_2-j_1+k\geqslant k_2$.
\end{enumerate}
Since $k\leqslant n$ the above considered cases imply that the  equation  $(i_1,j_1,[0;k_1])\cdot\chi=(i_2,j_2,[0;k_2])$ has finitely many solutions.

The proof of the statement that the  equation  $\chi\cdot(i_1,j_1,[0;k_1])=(i_2,j_2,[0;k_2])$ has finitely many solutions is similar.
\end{proof}

\begin{theorem}\label{theorem-2.5}
For an arbitrary $n\in\omega$ the semigroup $\boldsymbol{B}_{\omega}^{\mathscr{F}_n}$ is isomorphic to an inverse subsemigroup  of $\mathscr{I}_\omega^{n+1}$,
namely $\boldsymbol{B}_{\omega}^{\mathscr{F}_n}$ is isomorphic to the semigroup $\mathscr{I}_\omega^{n+1}(\overrightarrow{\mathrm{conv}})$.
\end{theorem}

\begin{proof}
We define a map $\mathfrak{I}\colon \boldsymbol{B}_{\omega}^{\mathscr{F}_n}\to \mathscr{I}_\omega^{n+1}$  by the formulae $\mathfrak{I}(\boldsymbol{0})=\boldsymbol{0}$ and
\begin{equation*}
  \mathfrak{I}(i,j,[0;k])=
  \left(
\begin{smallmatrix}
  i & i+1 & \cdots & i+k \\
  j & j+1 & \cdots & j+k \\
\end{smallmatrix}%
\right),
\qquad \hbox{for all} \quad i,j\in\omega \quad \hbox{and} \quad k=0,1,\ldots,n.
\end{equation*}
It is obvious that so defined map $\mathfrak{I}$ is injective.

\smallskip

Next we shall show that $\mathfrak{I}\colon \boldsymbol{B}_{\omega}^{\mathscr{F}_n}\to \mathscr{I}_\omega^{n+1}$ is a homomorphism.

\smallskip

It is obvious that
\begin{equation*}
  \mathfrak{I}(\boldsymbol{0}\cdot\boldsymbol{0})=\mathfrak{I}(\boldsymbol{0})=\boldsymbol{0}=\boldsymbol{0}\cdot\boldsymbol{0}= \mathfrak{I}(\boldsymbol{0})\cdot\mathfrak{I}(\boldsymbol{0}),
\end{equation*}
\begin{equation*}
  \mathfrak{I}(\boldsymbol{0}\cdot(i,j,[0;k]))=\mathfrak{I}(\boldsymbol{0})=\boldsymbol{0}=\boldsymbol{0}\cdot
  \left(
\begin{smallmatrix}
  i & i+1 & \cdots & i+k \\
  j & j+1 & \cdots & j+k \\
\end{smallmatrix}%
 \right)=
\mathfrak{I}(\boldsymbol{0})\cdot\mathfrak{I}(i,j,[0;k]),
\end{equation*}
and
\begin{equation*}
  \mathfrak{I}((i,j,[0;k])\cdot\boldsymbol{0})=\mathfrak{I}(\boldsymbol{0})=\boldsymbol{0}=
  \left(
\begin{smallmatrix}
  i & i+1 & \cdots & i+k \\
  j & j+1 & \cdots & j+k \\
\end{smallmatrix}%
 \right)
\cdot\boldsymbol{0}=
\mathfrak{I}(i,j,[0;k])\cdot\mathfrak{I}(\boldsymbol{0}),
\end{equation*}
for any non-zero element $(i,j,[0;k])$ of the semigroup $\boldsymbol{B}_{\omega}^{\mathscr{F}_n}$.

Fix arbitrary $i_1,i_2,j_1,j_2\in\omega$ and $k_1,k_2\in\{0,\dots,n\}$. In the case when $k_1\leqslant k_2$ we have that
\begin{align*}
  \mathfrak{I}((i_1,j_1,[0;k_1])&\cdot(i_2,j_2,[0;k_2]))=
\left\{
  \begin{array}{ll}
    \mathfrak{I}(i_1-j_1+i_2,j_2,(j_1-i_2+[0;k_1])\cap[0;k_2]), & \hbox{if}~ j_1<i_2;\\
    \mathfrak{I}(i_1,j_2,[0;k_1]\cap[0;k_2]),                   & \hbox{if}~ j_1=i_2;\\
    \mathfrak{I}(i_1,j_1-i_2+j_2,[0;k_1]\cap(i_2-j_1+[0;k_2])), & \hbox{if}~ j_1>i_2
  \end{array}
\right.=\\
  =&
\left\{
  \begin{array}{lll}
    \mathfrak{I}(\boldsymbol{0}),                   & \hbox{if}~ j_1<i_2 &\hbox{and}~~ j_1-i_2+k_1<0; \\
    \mathfrak{I}(i_1-j_1+i_2,j_2,[0;0]),            & \hbox{if}~ j_1<i_2 &\hbox{and}~~ j_1-i_2+k_1=0; \\
    \mathfrak{I}(i_1-j_1+i_2,j_2,[0;j_1-i_2+k_1]),  & \hbox{if}~ j_1<i_2 &\hbox{and}~~ 1\leqslant j_1-i_2+k_1\leqslant k_2;\\
    \mathfrak{I}(i_1,j_2,[0;k_1]),                  & \hbox{if}~ j_1=i_2;\\
    \mathfrak{I}(i_1,j_1-i_2+j_2,[0;k_1]),          & \hbox{if}~ j_1>i_2  &\hbox{and}~~ k_1\leqslant i_1-j_1+k_2;\\
    \mathfrak{I}(i_1,j_1-i_2+j_2,[0;i_2-j_1+k_2]),  & \hbox{if}~ j_1>i_2  &\hbox{and}~~ k_1> i_1-j_1+k_2;\\
    \mathfrak{I}(i_1,j_1-i_2+j_2,[0;0]),            & \hbox{if}~ j_1>i_2  &\hbox{and}~~ j_1=i_2+k_2;\\
    \mathfrak{I}(\boldsymbol{0}),                   & \hbox{if}~ j_1>i_2  &\hbox{and}~~ j_1>i_2+k_2
\end{array}
\right.=\\
%
   =&
\left\{
  \begin{array}{cll}
    \boldsymbol{0},                                  & \hbox{if}~ j_1<i_2 &\hbox{and}~~ j_1-i_2+k_1<0; \\
    \left(
      \begin{smallmatrix}
        i_1-j_1+i_2 \\
        j_2 \\
      \end{smallmatrix}
    \right)
    ,                                  & \hbox{if}~ j_1<i_2 &\hbox{and}~~ j_1-i_2+k_1=0; \\
    \left(
      \begin{smallmatrix}
        i_1-j_1+i_2 & \cdots & i_1+k_1 \\
        j_2         & \cdots & j_2+j_1-i_2+k_1 \\
      \end{smallmatrix}
    \right),    & \hbox{if}~ j_1<i_2 &\hbox{and}~~ 1\leqslant j_1-i_2+k_1\leqslant k_2;\\
    \left(
      \begin{smallmatrix}
        i_1 & \cdots & i_1+k_1 \\
        j_2 & \cdots & j_2+k_1 \\
      \end{smallmatrix}
    \right),                                & \hbox{if}~ j_1=i_2;\\
    \left(
      \begin{smallmatrix}
        i_1         & \cdots & i_1+k_1 \\
        j_1-i_2+j_2 & \cdots & j_1-i_2+j_2+k_1 \\
      \end{smallmatrix}
    \right),                                & \hbox{if}~ j_1>i_2  &\hbox{and}~~ k_1\leqslant i_2-j_1+k_2;\\
    \left(
      \begin{smallmatrix}
        i_1         & \cdots & i_1+i_2-j_1+k_2 \\
        j_1-i_2+j_2 & \cdots & j_2+k_2 \\
      \end{smallmatrix}
    \right),                                & \hbox{if}~ j_1>i_2  &\hbox{and}~~ k_1> i_2-j_1+k_2;\\
    \left(
      \begin{smallmatrix}
        i_1         \\
        j_1-i_2+j_2 \\
      \end{smallmatrix}
    \right),                               & \hbox{if}~ j_1>i_2  &\hbox{and}~~ j_1=i_2+k_2;\\
    \boldsymbol{0},                            & \hbox{if}~ j_1>i_2  &\hbox{and}~~ j_1>i_2+k_2
\end{array}
\right.=\\
%
  =&
\left\{
  \begin{array}{cll}
    \boldsymbol{0},                             & \hbox{if}~ j_1<i_2 &\hbox{and}~~ j_1-i_2+k_1<0; \\
    \left(
      \begin{smallmatrix}
        i_1+k_1 \\
        j_2 \\
      \end{smallmatrix}
    \right)
    ,                                           & \hbox{if}~ j_1<i_2 &\hbox{and}~~ j_1-i_2+k_1=0; \\
    \left(
      \begin{smallmatrix}
        i_1-j_1+i_2 & \cdots & i_1+k_1 \\
        j_2         & \cdots & j_2+j_1-i_2+k_1 \\
      \end{smallmatrix}
    \right),                                & \hbox{if}~ j_1<i_2 &\hbox{and}~~ 1\leqslant j_1-i_2+k_1\leqslant k_2;\\
    \left(
      \begin{smallmatrix}
        i_1 & \cdots & i_1+k_1 \\
        j_2 & \cdots & j_2+k_1 \\
      \end{smallmatrix}
    \right),                                & \hbox{if}~ j_1=i_2;\\
    \left(
      \begin{smallmatrix}
        i_1         & \cdots & i_1+k_1 \\
        j_1-i_2+j_2 & \cdots & j_1-i_2+j_2+k_1 \\
      \end{smallmatrix}
     \right),                                & \hbox{if}~ j_1>i_2  &\hbox{and}~~ k_1\leqslant i_2-j_1+k_2;\\
    \left(
      \begin{smallmatrix}
        i_1         & \cdots & i_1+i_2-j_1+k_2 \\
        j_1-i_2+j_2 & \cdots & j_2+k_2 \\
      \end{smallmatrix}
    \right),                                & \hbox{if}~ j_1>i_2  &\hbox{and}~~ k_1> i_2-j_1+k_2;\\
    \left(
      \begin{smallmatrix}
        i_1         \\
        j_2+k_2 \\
      \end{smallmatrix}
     \right),                                & \hbox{if}~ j_1>i_2  &\hbox{and}~~ j_1=i_2+k_2;\\
    \boldsymbol{0},                             & \hbox{if}~ j_1>i_2  &\hbox{and}~~ j_1>i_2+k_2
\end{array}
\right.
\end{align*}
and
\begin{align*}
  \mathfrak{I}(i_1,j_1,[0;k_1])&\cdot\mathfrak{I}(i_2,j_2,[0;k_2])=
\left(
  \begin{smallmatrix}
    i_1 & \cdots & i_1+k_1 \\
    j_1 & \cdots & j_1+k_1 \\
  \end{smallmatrix}
 \right)
\cdot
\left(
  \begin{smallmatrix}
    i_2 & \cdots & i_2+k_2 \\
    j_2 & \cdots & j_2+k_2 \\
  \end{smallmatrix}
 \right)
=\\
  =&
\left\{
  \begin{array}{cll}
    \boldsymbol{0},                                  & \hbox{if}~ j_1<i_2 &\hbox{and}~~ j_1+k_1<i_2; \\
    \left(
      \begin{smallmatrix}
        i_1+k_1 \\
        j_2 \\
      \end{smallmatrix}
    \right),                                  & \hbox{if}~ j_1<i_2 & \hbox{and}~~ j_1+k_1=i_2; \\
    \left(
      \begin{smallmatrix}
        i_1-j_1+i_2 & \cdots & i_1+k_1 \\
        j_2         & \cdots & j_2+j_1-i_2+k_1 \\
      \end{smallmatrix}
    \right),                                  & \hbox{if}~ j_1<i_2 & \hbox{and}~~ j_1+k_1\geqslant i_2+1;\\
    \left(
      \begin{smallmatrix}
        i_1 & \cdots & i_1+k_1 \\
        j_2 & \cdots & j_2+k_1 \\
      \end{smallmatrix}
    \right),                                  & \hbox{if}~ j_1=i_2;\\
    \left(
      \begin{smallmatrix}
        i_1         & \cdots & i_1+k_1 \\
        j_1-i_2+j_2 & \cdots & j_1-i_2+j_2+k_1 \\
      \end{smallmatrix}
    \right),                                  & \hbox{if}~ j_1>i_2  &\hbox{and}~~ j_1+k_1\leqslant i_2+k_2;\\
    \left(
      \begin{smallmatrix}
        i_1         & \cdots & i_1-j_1+i_2+k_2 \\
        j_1-i_2+j_2 & \cdots & j_2+k_2 \\
      \end{smallmatrix}
    \right),                                  & \hbox{if}~ j_1>i_2  &\hbox{and}~~ j_1+k_1>i_2+k_2;\\
    \left(
      \begin{smallmatrix}
        i_1 \\
        j_2+k_2 \\
      \end{smallmatrix}
    \right),                                  & \hbox{if}~ j_1>i_2  &\hbox{and}~~ j_1=i_2+k_2;\\
    \boldsymbol{0},                               & \hbox{if}~ j_1>i_2  &\hbox{and}~~ j_1>i_2+k_2.
\end{array}
\right.
\end{align*}

In the case when $k_1\geqslant k_2$ we have that
\begin{align*}
  \mathfrak{I}((i_1,j_1,[0;k_1])&\cdot(i_2,j_2,[0;k_2]))=
\left\{
  \begin{array}{ll}
    \mathfrak{I}(i_1-j_1+i_2,j_2,(j_1-i_2+[0;k_1])\cap[0;k_2]), & \hbox{if}~ j_1<i_2;\\
    \mathfrak{I}(i_1,j_2,[0;k_1]\cap[0;k_2]),                   & \hbox{if}~ j_1=i_2;\\
    \mathfrak{I}(i_1,j_1-i_2+j_2,[0;k_1]\cap(i_2-j_1+[0;k_2])), & \hbox{if}~ j_1>i_2
  \end{array}
\right.=\\
  =&
\left\{
  \begin{array}{lll}
    \mathfrak{I}(\boldsymbol{0}),                   & \hbox{if}~ j_1<i_2 &\hbox{and}~~ j_1-i_2+k_1<0; \\
    \mathfrak{I}(i_1-j_1+i_2,j_2,[0;0]),            & \hbox{if}~ j_1<i_2 &\hbox{and}~~ j_1-i_2+k_1=0; \\
    \mathfrak{I}(i_1-j_1+i_2,j_2,[0;j_1-i_2+k_1]),  & \hbox{if}~ j_1<i_2 &\hbox{and}~~ 1\leqslant j_1+k_1\leqslant i_2+k_2;\\
    \mathfrak{I}(i_1-j_1+i_2,j_2,[0;k_2]),          & \hbox{if}~ j_1<i_2 &\hbox{and}~~ j_1+k_1> i_2+k_2;\\
    \mathfrak{I}(i_1,j_2,[0;k_2]),                  & \hbox{if}~ j_1=i_2;\\
    \mathfrak{I}(i_1,j_1-i_2+j_2,[0;i_2-j_1+k_2]),  & \hbox{if}~ j_1>i_2  &\hbox{and}~~ i_2-j_1+k_2>0;\\
    \mathfrak{I}(i_1,j_1-i_2+j_2,[0;0]),            & \hbox{if}~ j_1>i_2  &\hbox{and}~~ i_2-j_1+k_2=0;\\
    \mathfrak{I}(\boldsymbol{0}),                   & \hbox{if}~ j_1>i_2  &\hbox{and}~~ i_2-j_1+k_2<0
\end{array}
\right.=\\
%
  =&
\left\{
  \begin{array}{cll}
    \boldsymbol{0},                                      & \hbox{if}~ j_1<i_2 &\hbox{and}~~ j_1-i_2+k_1<0; \\
    \left(
      \begin{smallmatrix}
        i_1-j_1+i_2 \\
        j_2 \\
      \end{smallmatrix}
    \right)                                          & \hbox{if}~ j_1<i_2 &\hbox{and}~~ j_1-i_2+k_1=0; \\
    \left(
      \begin{smallmatrix}
        i_1-j_1+i_2 & \cdots & i_1+k_1\\
        j_2         & \cdots & j_1-i_2+j_2+k_1\\
      \end{smallmatrix}
    \right)                                          & \hbox{if}~ j_1<i_2 &\hbox{and}~~ j_1+k_1\leqslant i_2+k_2; \\
    \left(
      \begin{smallmatrix}
        i_1-j_1+i_2 & \cdots & i_1-j_1+i_2+k_2\\
        j_2         & \cdots & j_2+k_2\\
      \end{smallmatrix}
    \right)                                          & \hbox{if}~ j_1<i_2 &\hbox{and}~~ j_1+k_1> i_2+k_2; \\
    \left(
      \begin{smallmatrix}
        i_1+k_2 & \cdots & i_1+k_2\\
        j_2+k_2 & \cdots & j_2+k_2\\
      \end{smallmatrix}
    \right)                                          & \hbox{if}~ j_1=i_2; \\
    \left(
      \begin{smallmatrix}
        i_1         & \cdots & i_1-j_1+i_2+k_2\\
        j_1-i_2+j_2 & \cdots & j_2+k_2\\
      \end{smallmatrix}
    \right)                                          & \hbox{if}~ j_1>i_2 &\hbox{and}~~ i_2-j_1+k_2>0; \\
    \left(
      \begin{smallmatrix}
        i_1         \\
        j_1-i_2+j_2 \\
      \end{smallmatrix}
    \right)                                          & \hbox{if}~ j_1>i_2 &\hbox{and}~~ i_2-j_1+k_2=0; \\
    \boldsymbol{0},                                      & \hbox{if}~ j_1>i_2 &\hbox{and}~~ i_2-j_1+k_2<0
   \end{array}
\right.  =
\\
%
=&
\left\{
  \begin{array}{cll}
    \boldsymbol{0},                                  & \hbox{if}~ j_1<i_2 &\hbox{and}~~ j_1-i_2+k_1<0; \\
    \left(
      \begin{smallmatrix}
        i_1+k_1 \\
        j_2 \\
      \end{smallmatrix}
    \right)                                          & \hbox{if}~ j_1<i_2 &\hbox{and}~~ j_1-i_2+k_1=0; \\
    \left(
      \begin{smallmatrix}
        i_1-j_1+i_2 & \cdots & i_1+k_1\\
        j_2         & \cdots & j_1-i_2+j_2+k_1\\
      \end{smallmatrix}
     \right)                                          & \hbox{if}~ j_1<i_2 &\hbox{and}~~ j_1+k_1\leqslant i_2+k_2; \\
    \left(
      \begin{smallmatrix}
        i_1-j_1+i_2 & \cdots & i_1-j_1+i_2+k_2\\
        j_2         & \cdots & j_2+k_2\\
      \end{smallmatrix}
     \right)                                          & \hbox{if}~ j_1<i_2 &\hbox{and}~~ j_1+k_1> i_2+k_2; \\
    \left(
      \begin{smallmatrix}
        i_1 & \cdots & i_1+k_2\\
        j_2 & \cdots & j_2+k_2\\
      \end{smallmatrix}
    \right)                                          & \hbox{if}~ j_1=i_2; \\
    \left(
      \begin{smallmatrix}
        i_1         & \cdots & i_1-j_1+i_2+k_2\\
        j_1-i_2+j_2 & \cdots & j_2+k_2\\
      \end{smallmatrix}
    \right)                                          & \hbox{if}~ j_1>i_2 &\hbox{and}~~ j_1<i_2+k_2; \\
    \left(
      \begin{smallmatrix}
        i_1         \\
        j_2+k_2 \\
      \end{smallmatrix}
    \right)                                          & \hbox{if}~ j_1>i_2 &\hbox{and}~~ j_1=i_2+k_2; \\
    \boldsymbol{0},                                      & \hbox{if}~ j_1>i_2 &\hbox{and}~~ j_1>i_2+k_2.
   \end{array}
\right.
\end{align*}
and
\begin{align*}
  \mathfrak{I}(i_1,j_1,[0;k_1])&\cdot\mathfrak{I}(i_2,j_2,[0;k_2])=
\left(
  \begin{smallmatrix}
    i_1 & \cdots & i_1+k_1 \\
    j_1 & \cdots & j_1+k_1 \\
  \end{smallmatrix}
 \right)
\cdot
\left(
  \begin{smallmatrix}
    i_2 & \cdots & i_2+k_2 \\
    j_2 & \cdots & j_2+k_2 \\
  \end{smallmatrix}
 \right)
=\\
\end{align*}

\begin{align*}
=&
\left\{
  \begin{array}{cll}
    \boldsymbol{0},                                  & \hbox{if}~ j_1<i_2 &\hbox{and}~~ j_1+k_1<i_2; \\
    \left(
      \begin{smallmatrix}
        i_1+k_1 \\
        j_2 \\
      \end{smallmatrix}
    \right),                                  & \hbox{if}~ j_1<i_2 & \hbox{and}~~ j_1+k_1=i_2; \\
    \left(
      \begin{smallmatrix}
        i_1-j_1+i_2 & \cdots & i_1+k_1 \\
        j_2         & \cdots & j_2+j_1-i_2+k_1 \\
      \end{smallmatrix}
    \right),                                  & \hbox{if}~ j_1<i_2 & \hbox{and}~~ j_1+k_1\leqslant i_2+k_2;\\
    \left(
      \begin{smallmatrix}
        i_1-j_1+i_2 & \cdots &  i_1-j_1+i_2+k_2 \\
        j_2         & \cdots & j_2+k_2 \\
      \end{smallmatrix}
     \right),                                  & \hbox{if}~ j_1<i_2 & \hbox{and}~~ j_1+k_1> i_2+k_2;\\
    \left(
      \begin{smallmatrix}
        i_1 & \cdots & i_1+k_2 \\
        j_2 & \cdots & j_2+k_2 \\
      \end{smallmatrix}
    \right),                                  & \hbox{if}~ j_1=i_2;\\
    \left(
      \begin{smallmatrix}
        i_1         & \cdots & i_1-j_1+i_2+k_2 \\
        j_1-i_2+j_2 & \cdots & i_2+k_2 \\
      \end{smallmatrix}
    \right),                                  & \hbox{if}~ j_1>i_2  &\hbox{and}~~ j_1< i_2+k_2;\\
    \left(
      \begin{smallmatrix}
        i_1 \\
        j_2+k_2 \\
      \end{smallmatrix}
    \right),                                  & \hbox{if}~ j_1>i_2  &\hbox{and}~~ j_1=i_2+k_2;\\
    \boldsymbol{0},                               & \hbox{if}~ j_1>i_2  &\hbox{and}~~ j_1>i_2+k_2.
\end{array}
\right.
\end{align*}

By Lemma~II.1.10 of \cite{Petrich-1984} the homomorphic image $\mathfrak{I}(\boldsymbol{B}_{\omega}^{\mathscr{F}_n})$ is an inverse subsemigroup of $\mathscr{I}_\omega^{n+1}$.

It is obvious that $\mathfrak{I}(\boldsymbol{0})$ is the empty partial self-map of $\omega$ and it is by the assumption is an order convex partial isomorphism of $(\omega,\leqslant)$. Also the image
\begin{equation*}
  \mathfrak{I}(i,j,[0;k])=
  \left( %
\begin{smallmatrix}
  i & i+1 & \cdots & i+k \\
  j & j+1 & \cdots & j+k \\
\end{smallmatrix}%
 \right)
\end{equation*}
is an order convex partial isomorphism of $(\omega,\leqslant)$ for all $i,j\in\omega$ and $k=0,1,\ldots,n$. The definition of $\mathfrak{I}\colon \boldsymbol{B}_{\omega}^{\mathscr{F}_n}\to \mathscr{I}_\omega^{n+1}$ implies that its co-restriction on the image $\mathscr{I}_\omega^{n+1}(\overrightarrow{\mathrm{conv}})$ is surjective, and hence $\mathfrak{I}\colon \boldsymbol{B}_{\omega}^{\mathscr{F}_n}\to \mathscr{I}_\omega^{n+1}(\overrightarrow{\mathrm{conv}})$ is an isomorphism.
\end{proof}

\begin{remark}\label{remark-2.6}
Observe that the image $\mathfrak{I}(\boldsymbol{B}_{\omega}^{\mathscr{F}_n})$ does not contains all idempotents of the semigroup $\mathscr{I}_\omega^{n+1}$, especially
$
\left(
  \begin{smallmatrix}
    0 & 2 \\
    0 & 2 \\
  \end{smallmatrix}
\right)\notin \mathfrak{I}(\boldsymbol{B}_{\omega}^{\mathscr{F}_n})
$ for any $n\geqslant 1$. But by Proposition~4 of \cite{Gutik-Mykhalenych=2020} the semigroup $\boldsymbol{B}_{\omega}^{\mathscr{F}_0}$ is isomorphic to the semigroup $\omega{\times}\omega$-matrix units, and hence $\boldsymbol{B}_{\omega}^{\mathscr{F}_0}$ is isomorphic to the semigroup $\mathscr{I}_\omega^{1}$
\end{remark}

A subset $D$ of a semigroup $S$ is said to be \emph{$\omega$-unstable} if $D$ is infinite and for any $a\in D$ and an infinite subset $B\subseteq D$, we have $aB\cup Ba\nsubseteq D$ \cite{Gutik-Lawson-Repov=2009}. A basic example of $\omega$-unstable sets is given in \cite{Gutik-Lawson-Repov=2009}: for an infinite cardinal $\lambda$ the set $D=\mathscr{I}_\omega^{n}\setminus \mathscr{I}_\omega^{n-1}$  is an $\omega$-unstable subset of $\mathscr{I}_\omega^{n}$.

For any $n\in\omega$ the definition of the semigroup operation on $\boldsymbol{B}_{\omega}^{\mathscr{F}_n}$ implies that its subsemigroup $\boldsymbol{B}_{\omega}^{\mathscr{F}_k}$ is an ideal of $\boldsymbol{B}_{\omega}^{\mathscr{F}_n}$ for any $k\in\{0,\ldots,n\}$. Also, since $\mathscr{I}_\omega^{k+1}(\overrightarrow{\mathrm{conv}})\setminus \mathscr{I}_\omega^{k}(\overrightarrow{\mathrm{conv}})$ is an infinite subset of $\mathscr{I}_\omega^{n+1}(\overrightarrow{\mathrm{conv}})$ for any $k\in\{0,\ldots,n\}$, the above arguments and Theorem~\ref{theorem-2.5} imply the following lemma:

\begin{lemma}\label{lemma-2.7}
For an arbitrary $n\in\omega$ the subsets $\boldsymbol{B}_{\omega}^{\mathscr{F}_0}\setminus\{\boldsymbol{0}\}$ and $\boldsymbol{B}_{\omega}^{\mathscr{F}_{k}}\setminus \boldsymbol{B}_{\omega}^{\mathscr{F}_{k-1}}$ are $\omega$-unstable of $\boldsymbol{B}_{\omega}^{\mathscr{F}_{n}}$ for any $k\in\{1,\ldots,n\}$.
\end{lemma}

\begin{proof}
We shall show that the set $\boldsymbol{B}_{\omega}^{\mathscr{F}_{k}}\setminus \boldsymbol{B}_{\omega}^{\mathscr{F}_{k-1}}$ is $\omega$-unstable, and the proof that the set $\boldsymbol{B}_{\omega}^{\mathscr{F}_0}\setminus\{\boldsymbol{0}\}$ is $\omega$-unstable is similar.

Fix an arbitrary distinct $(i_1,j_1,[0;k]),(i_2,j_2,[0;k])\in\boldsymbol{B}_{\omega}^{\mathscr{F}_{k}}\setminus \boldsymbol{B}_{\omega}^{\mathscr{F}_{k-1}}$. The definition of the semigroup operation of $\boldsymbol{B}_{\omega}^{\mathscr{F}_{n}}$ implies that for any $(i,j,[0;k])\in\boldsymbol{B}_{\omega}^{\mathscr{F}_{k}}\setminus \boldsymbol{B}_{\omega}^{\mathscr{F}_{k-1}}$ we have that
\begin{equation*}
(i,j,[0;k])\cdot(i_p,j_p,[0;k])=
   \left\{
     \begin{array}{ll}
       (i-j+i_p,j_p,(j-i_p+[0;k])\cap[0;k]), & \hbox{if~} j<i_p;\\
       (i,j_p,[0;k]\cap[0;k]),               & \hbox{if~} j=i_p;\\
       (i,j-i_p+j_p,[0;k]\cap(i_p-j+[0;k])), & \hbox{if~} j>i_p
     \end{array}
   \right.
\end{equation*}
for $p=1,2$. In the case when $i_1\neq i_2$ we obtain that $(i,j,[0;k])\cdot\left\{(i_1,j_1,[0;k]),(i_2,j_2,[0;k])\right\}\nsubseteq \boldsymbol{B}_{\omega}^{\mathscr{F}_{k}}\setminus \boldsymbol{B}_{\omega}^{\mathscr{F}_{k-1}}$. In the case when $j_1\neq j_2$ the proof is similar.
\end{proof}

\begin{definition}[{\!\!\cite{Gutik-Lawson-Repov=2009}}]\label{definition-2.8}
An ideal series  for a semigroup $S$ is a chain of ideals
\begin{equation*}
  I_0\subseteq I_1 \subseteq I_2\subseteq\cdots\subseteq I_m =S.
\end{equation*}
This ideal series is called \emph{tight} if $I_0$ is a finite set and $D_k=I_k\setminus I_{k-1}$ is an $\omega$-unstable subset for each $k=1,\ldots,m$.
\end{definition}

Lemma~\ref{lemma-2.7} implies

\begin{proposition}\label{proposition-2.9}
For an arbitrary $n\in\omega$ the following ideal series
\begin{equation*}
  \{\boldsymbol{0}\}\subseteq \boldsymbol{B}_{\omega}^{\mathscr{F}_0}\subseteq \boldsymbol{B}_{\omega}^{\mathscr{F}_1}\subseteq \cdots \subseteq \boldsymbol{B}_{\omega}^{\mathscr{F}_{n-1}}\subseteq \boldsymbol{B}_{\omega}^{\mathscr{F}_n}
\end{equation*}
is tight.
\end{proposition}

\begin{proposition}\label{proposition-2.10}
For any non-negative integer $n$ and arbitrary $p=0,1,\ldots,n-1$ the map $\mathfrak{h}_p\colon \boldsymbol{B}_{\omega}^{\mathscr{F}_n}\to \boldsymbol{B}_{\omega}^{\mathscr{F}_n}$ defined by the formulae $\mathfrak{h}_p(\boldsymbol{0})=\boldsymbol{0}$ and
\begin{equation*}
\mathfrak{h}_p(i,j,[0;k])=
\left\{
  \begin{array}{cl}
    \boldsymbol{0}, & \hbox{if}~ k=0,1,\dots,p;\\
    (i,j,[0;k-p-1]),  & \hbox{if}~ k=p+1,\ldots,n,
  \end{array}
\right.
\end{equation*}
is a homomorphism which maps the semigroup $\boldsymbol{B}_{\omega}^{\mathscr{F}_n}$ onto its subsemigroup $\boldsymbol{B}_{\omega}^{\mathscr{F}_{n-p-1}}$.
\end{proposition}

\begin{proof}
First we shall show that the map $\mathfrak{h}_0\colon \boldsymbol{B}_{\omega}^{\mathscr{F}_n}\to \boldsymbol{B}_{\omega}^{\mathscr{F}_n}$ defined by the formulae $\mathfrak{h}_0(\boldsymbol{0})=\boldsymbol{0}$ and
\begin{equation*}
\mathfrak{h}_0(i,j,[0;k])=
\left\{
  \begin{array}{cl}
    \boldsymbol{0}, & \hbox{if}~ k=0;\\
    (i,j,[0;k-1]),  & \hbox{if}~ k=1,\ldots,n,
  \end{array}
\right.
\end{equation*}
is a homomorphism.

It is obvious that
\begin{equation*}
  \mathfrak{h}_0(\boldsymbol{0})\cdot\mathfrak{h}_0(i,j,[0])=\boldsymbol{0}\cdot \boldsymbol{0}=\mathfrak{h}_0(\boldsymbol{0})=\mathfrak{h}_0(\boldsymbol{0}\cdot (i,j,[0]))
\end{equation*}
and
\begin{equation*}
  \mathfrak{h}_0(i,j,[0])\cdot\mathfrak{h}_0(\boldsymbol{0})=\boldsymbol{0}\cdot \boldsymbol{0}=\mathfrak{h}_0(\boldsymbol{0})=\mathfrak{h}_0((i,j,[0])\cdot \boldsymbol{0})
\end{equation*}
for any $i,j\in\omega$.

Fix arbitrary $i_1,i_2,j_1,j_2\in\omega$ and positive integers $k_1$ and $k_2$. In the case when $k_1\leqslant k_2$ we have that
\begin{align*}
  \mathfrak{h}_0(i_1,j_1,&[0;k_1])\cdot\mathfrak{h}_0(i_2,j_2,[0;k_2])=(i_1,j_1,[0;k_1-1])\cdot(i_2,j_2,[0;k_2-1])= \\
  =&
\left\{
  \begin{array}{ll}
    (i_1-j_1+i_2,j_2,(j_1-i_2+[0;k_1-1])\cap[0;k_2-1]), & \hbox{if}~ j_1<i_2;\\
    (i_1,j_2,[0;k_1-1]\cap[0;k_2-1]),                   & \hbox{if}~ j_1=i_2;\\
    (i_1,j_1-i_2+j_2,[0;k_1-1]\cap(i_2-j_1+[0;k_2-1])), & \hbox{if}~ j_1>i_2
  \end{array}
\right.=\\
  =&
\left\{
  \begin{array}{lll}
    \boldsymbol{0},                                     & \hbox{if}~ j_1<i_2 &\hbox{and}~~ j_1-i_2+k_1-1<0; \\
    (i_1-j_1+i_2,j_2,[0;j_1-i_2+k_1-1]),                & \hbox{if}~ j_1<i_2 &\hbox{and}~~ 0\leqslant j_1-i_2+k_1-1\leqslant k_2-1;\\
    (i_1,j_2,[0;k_1-1]),                                & \hbox{if}~ j_1=i_2;\\
    (i_1,j_1-i_2+j_2,[0;k_1-1]),                        & \hbox{if}~ j_1>i_2  &\hbox{and}~~ k_1-1\leqslant i_1-j_1+k_2-1;\\
    (i_1,j_1-i_2+j_2,[0;i_2-j_1+k_2-1]),                & \hbox{if}~ j_1>i_2  &\hbox{and}~~ k_1-1> i_1-j_1+k_2-1\geqslant 0;\\
    \boldsymbol{0},                                     & \hbox{if}~ j_1>i_2  &\hbox{and}~~ k_1-1> i_1-j_1+k_2-1< 0
  \end{array}
\right.=\\
  =&
\left\{
  \begin{array}{lll}
    \boldsymbol{0},                                     & \hbox{if}~ j_1<i_2 &\hbox{and}~~ j_1-i_2+k_1<1; \\
    (i_1-j_1+i_2,j_2,[0;j_1-i_2+k_1-1]),                & \hbox{if}~ j_1<i_2 &\hbox{and}~~ 1\leqslant j_1-i_2+k_1\leqslant k_2;\\
    (i_1,j_2,[0;k_1-1]),                                & \hbox{if}~ j_1=i_2;\\
    (i_1,j_1-i_2+j_2,[0;k_1-1]),                        & \hbox{if}~ j_1>i_2  &\hbox{and}~~ k_1\leqslant i_2-j_1+k_2;\\
    (i_1,j_1-i_2+j_2,[0;i_2-j_1+k_2-1]),                & \hbox{if}~ j_1>i_2  &\hbox{and}~~ k_1> i_2-j_1+k_2\geqslant 1;\\
    \boldsymbol{0},                                     & \hbox{if}~ j_1>i_2  &\hbox{and}~~ k_1> i_2-j_1+k_2<1,
  \end{array}
\right.
\end{align*}
and
\begin{align*}
  \mathfrak{h}_0((i_1,j_1,[0;k_1])&\cdot(i_2,j_2,[0;k_2]))=
\left\{
  \begin{array}{ll}
    \mathfrak{h}_0(i_1-j_1+i_2,j_2,(j_1-i_2+[0;k_1])\cap[0;k_2]), & \hbox{if}~ j_1<i_2;\\
    \mathfrak{h}_0(i_1,j_2,[0;k_1]\cap[0;k_2]),                   & \hbox{if}~ j_1=i_2;\\
    \mathfrak{h}_0(i_1,j_1-i_2+j_2,[0;k_1]\cap(i_2-j_1+[0;k_2])), & \hbox{if}~ j_1>i_2
  \end{array}
\right.=
\\
%
  =&
\left\{
  \begin{array}{lll}
    \mathfrak{h}_0(\boldsymbol{0}),                   & \hbox{if}~ j_1<i_2 &\hbox{and}~~ j_1-i_2+k_1<0; \\
    \mathfrak{h}_0(i_1-j_1+i_2,j_2,[0;0]\cap[0;k_2]), & \hbox{if}~ j_1<i_2 &\hbox{and}~~ j_1-i_2+k_1=0; \\
    \mathfrak{h}_0(i_1-j_1+i_2,j_2,[0;j_1-i_2+k_1]),  & \hbox{if}~ j_1<i_2 &\hbox{and}~~ 1\leqslant j_1-i_2+k_1\leqslant k_2;\\
    \mathfrak{h}_0(i_1,j_2,[0;k_1]),                  & \hbox{if}~ j_1=i_2;\\
    \mathfrak{h}_0(i_1,j_1-i_2+j_2,[0;k_1]),          & \hbox{if}~ j_1>i_2  &\hbox{and}~~ k_1\leqslant i_2-j_1+k_2;\\
    \mathfrak{h}_0(i_1,j_1-i_2+j_2,[0;k_1]\cap[0;0]), & \hbox{if}~ j_1>i_2  &\hbox{and}~~ k_1> i_2-j_1+k_2=0;\\
    \mathfrak{h}_0(\boldsymbol{0}),                   & \hbox{if}~ j_1>i_2  &\hbox{and}~~ k_1> i_2-j_1+k_2<0
\end{array}
\right.=
\end{align*}

\begin{align*}
  =&
\left\{
  \begin{array}{lll}
    \mathfrak{h}_0(\boldsymbol{0}),                  & \hbox{if}~ j_1<i_2 &\hbox{and}~~ j_1-i_2+k_1<0; \\
    \mathfrak{h}_0(i_1-j_1+i_2,j_2,[0;0]),           & \hbox{if}~ j_1<i_2 &\hbox{and}~~ j_1-i_2+k_1=0; \\
    \mathfrak{h}_0(i_1-j_1+i_2,j_2,[0;j_1-i_2+k_1]), & \hbox{if}~ j_1<i_2 &\hbox{and}~~ 1\leqslant j_1-i_2+k_1\leqslant k_2;\\
    \mathfrak{h}_0(i_1,j_2,[0;k_1]),                 & \hbox{if}~ j_1=i_2;\\
    \mathfrak{h}_0(i_1,j_1-i_2+j_2,[0;k_1]),         & \hbox{if}~ j_1>i_2  &\hbox{and}~~ k_1\leqslant i_2-j_1+k_2;\\
    \mathfrak{h}_0(i_1,j_1-i_2+j_2,[0;0]),           & \hbox{if}~ j_1>i_2  &\hbox{and}~~ k_1> i_2-j_1+k_2=0;\\
    \mathfrak{h}_0(\boldsymbol{0}),                  & \hbox{if}~ j_1>i_2  &\hbox{and}~~ k_1> i_2-j_1+k_2<0
\end{array}
\right.=\\
  =&
\left\{
  \begin{array}{lll}
    \boldsymbol{0},                                  & \hbox{if}~ j_1<i_2 &\hbox{and}~~ j_1-i_2+k_1<0; \\
    \boldsymbol{0},                                  & \hbox{if}~ j_1<i_2 &\hbox{and}~~ j_1-i_2+k_1=0; \\
    \mathfrak{h}_0(i_1-j_1+i_2,j_2,[0;j_1-i_2+k_1]), & \hbox{if}~ j_1<i_2 &\hbox{and}~~ 1\leqslant j_1-i_2+k_1\leqslant k_2;\\
    \mathfrak{h}_0(i_1,j_2,[0;k_1]),                 & \hbox{if}~ j_1=i_2;\\
    \mathfrak{h}_0(i_1,j_1-i_2+j_2,[0;k_1]),         & \hbox{if}~ j_1>i_2  &\hbox{and}~~ k_1\leqslant i_2-j_1+k_2;\\
    \boldsymbol{0},                                  & \hbox{if}~ j_1>i_2  &\hbox{and}~~ k_1> i_2-j_1+k_2=0;\\
    \boldsymbol{0},                                  & \hbox{if}~ j_1>i_2  &\hbox{and}~~ k_1> i_2-j_1+k_2<0
\end{array}
\right.=\\
%
  =&
\left\{
  \begin{array}{lll}
    \boldsymbol{0},                                     & \hbox{if}~ j_1<i_2 &\hbox{and}~~ j_1-i_2+k_1<1; \\
    (i_1-j_1+i_2,j_2,[0;j_1-i_2+k_1-1]),                & \hbox{if}~ j_1<i_2 &\hbox{and}~~ 1\leqslant j_1-i_2+k_1\leqslant k_2;\\
    (i_1,j_2,[0;k_1-1]),                                & \hbox{if}~ j_1=i_2;\\
    (i_1,j_1-i_2+j_2,[0;k_1-1]),                        & \hbox{if}~ j_1>i_2  &\hbox{and}~~ k_1\leqslant i_2-j_1+k_2;\\
    (i_1,j_1-i_2+j_2,[0;i_2-j_1+k_2-1]),                & \hbox{if}~ j_1>i_2  &\hbox{and}~~ k_1> i_2-j_1+k_2\geqslant 1;\\
    \boldsymbol{0},                                     & \hbox{if}~ j_1>i_2  &\hbox{and}~~ k_1> i_2-j_1+k_2<1,
  \end{array}
\right.
\end{align*}

In the case when $k_1\geqslant k_2$ we have that
\begin{align*}
  \mathfrak{h}_0&(i_1,j_1,[0;k_1])\cdot\mathfrak{h}_0(i_2,j_2,[0;k_2])=(i_1,j_1,[0;k_1-1])\cdot(i_2,j_2,[0;k_2-1])= \\
  =&
\left\{
  \begin{array}{ll}
    (i_1-j_1+i_2,j_2,(j_1-i_2+[0;k_1-1])\cap[0;k_2-1]), & \hbox{if}~ j_1<i_2;\\
    (i_1,j_2,[0;k_1-1]\cap[0;k_2-1]),                   & \hbox{if}~ j_1=i_2;\\
    (i_1,j_1-i_2+j_2,[0;k_1-1]\cap(i_2-j_1+[0;k_2-1])), & \hbox{if}~ j_1>i_2
  \end{array}
\right.=\\
  =&
\left\{
  \begin{array}{lll}
    \boldsymbol{0},                                          & \hbox{if}~ j_1<i_2 &\hbox{and}~~ j_1-i_2+k_1-1<0; \\
    (i_1-j_1+i_2,j_2,[0;0]\cap[0;k_2-1]),                    & \hbox{if}~ j_1<i_2 &\hbox{and}~~ j_1-i_2+k_1-1=0; \\
    (i_1-j_1+i_2,j_2,[0;j_1-i_2+k_1{-}1]{\cap}[0;k_2{-}1]),  & \hbox{if}~ j_1<i_2 &\hbox{and}~~ 1\leqslant j_1-i_2+k_1-1\leqslant k_2-1; \\
    (i_1-j_1+i_2,j_2,[0;k_2-1]),                             & \hbox{if}~ j_1<i_2 &\hbox{and}~~ k_2-1<j_1-i_2+k_1-1; \\
    (i_1,j_2,[0;k_1-1]\cap[0;k_2-1]),                        & \hbox{if}~ j_1=i_2;\\
    (i_1,j_1-i_2+j_2,[0;i_2-j_1+k_2-1]),                     & \hbox{if}~ j_1>i_2 &\hbox{and}~~ i_2-j_1+k_2-1\geqslant 0;\\
    \boldsymbol{0},                                          & \hbox{if}~ j_1>i_2 &\hbox{and}~~ i_2-j_1+k_2-1< 0
  \end{array}
\right.=
\\
%
  =&
\left\{
  \begin{array}{lll}
    \boldsymbol{0},                                      & \hbox{if}~ j_1<i_2 &\hbox{and}~~ j_1-i_2+k_1<1; \\
    (i_1-j_1+i_2,j_2,[0;0]),                             & \hbox{if}~ j_1<i_2 &\hbox{and}~~ j_1-i_2+k_1=1; \\
    (i_1-j_1+i_2,j_2,[0;j_1-i_2+k_1-1]),                 & \hbox{if}~ j_1<i_2 &\hbox{and}~~ 1\leqslant j_1-i_2+k_1-1\leqslant k_2-1; \\
    (i_1-j_1+i_2,j_2,[0;k_2-1]),                         & \hbox{if}~ j_1<i_2 &\hbox{and}~~ k_2<j_1-i_2+k_1; \\
    (i_1,j_2,[0;k_2-1]),                                 & \hbox{if}~ j_1=i_2;\\
    (i_1,j_1-i_2+j_2,[0;i_2-j_1+k_2-1]),                 & \hbox{if}~ j_1>i_2 &\hbox{and}~~ i_2-j_1+k_2\geqslant 1;\\
    \boldsymbol{0},                                      & \hbox{if}~ j_1>i_2 &\hbox{and}~~ i_2-j_1+k_2<1
  \end{array}
\right.
\end{align*}
and
\begin{align*}
  \mathfrak{h}_0((i_1,j_1,[0;k_1])&\cdot(i_2,j_2,[0;k_2]))=
\left\{
  \begin{array}{ll}
    \mathfrak{h}_0(i_1-j_1+i_2,j_2,(j_1-i_2+[0;k_1])\cap[0;k_2]), & \hbox{if}~ j_1<i_2;\\
    \mathfrak{h}_0(i_1,j_2,[0;k_1]\cap[0;k_2]),                   & \hbox{if}~ j_1=i_2;\\
    \mathfrak{h}_0(i_1,j_1-i_2+j_2,[0;k_1]\cap(i_2-j_1+[0;k_2])), & \hbox{if}~ j_1>i_2
  \end{array}
\right.=
\end{align*}

\begin{align*}
  =&
\left\{
  \begin{array}{lll}
    \mathfrak{h}_0(\boldsymbol{0}),                                       & \hbox{if}~ j_1<i_2 &\hbox{and}~~ j_1-i_2+k_1<0; \\
    \mathfrak{h}_0(i_1-j_1+i_2,j_2,[0;0]),                                & \hbox{if}~ j_1<i_2 &\hbox{and}~~ j_1-i_2+k_1=0; \\
    \mathfrak{h}_0(i_1-j_1+i_2,j_2,[0;j_1-i_2+k_1]),                      & \hbox{if}~ j_1<i_2 &\hbox{and}~~ k_2\geqslant j_1-i_2+k_1\geqslant 1; \\
    \mathfrak{h}_0(i_1-j_1+i_2,j_2,[0;k_2]),                              & \hbox{if}~ j_1<i_2 &\hbox{and}~~ k_2< j_1-i_2+k_1\geqslant 1; \\
    \mathfrak{h}_0(i_1,j_2,[0;k_2]),                                      & \hbox{if}~ j_1=i_2;\\
    \mathfrak{h}_0(i_1,j_1-i_2+j_2,[0;i_2-j_1+k_2]),                      & \hbox{if}~ j_1>i_2 &\hbox{and}~~ 1\leqslant i_2-j_1+k_2;\\
    \mathfrak{h}_0(i_1,j_1-i_2+j_2,[0;0]),                                & \hbox{if}~ j_1>i_2 &\hbox{and}~~ 0=i_2-j_1+k_2;\\
    \mathfrak{h}_0(\boldsymbol{0}),                                       & \hbox{if}~ j_1>i_2 &\hbox{and}~~ i_2-j_1+k_2<0
  \end{array}
\right.=\\
%
  =&
\left\{
  \begin{array}{lll}
    \mathfrak{h}_0(\boldsymbol{0}),                                       & \hbox{if}~ j_1<i_2 &\hbox{and}~~ j_1-i_2+k_1<0; \\
    \mathfrak{h}_0(\boldsymbol{0}),                                       & \hbox{if}~ j_1<i_2 &\hbox{and}~~ j_1-i_2+k_1=0; \\
    \mathfrak{h}_0(i_1-j_1+i_2,j_2,[0;j_1-i_2+k_1]),                      & \hbox{if}~ j_1<i_2 &\hbox{and}~~ k_2\geqslant j_1-i_2+k_1\geqslant 1; \\
    \mathfrak{h}_0(i_1-j_1+i_2,j_2,[0;k_2]),                              & \hbox{if}~ j_1<i_2 &\hbox{and}~~ k_2< j_1-i_2+k_1\geqslant 1; \\
    \mathfrak{h}_0(i_1,j_2,[0;k_2]),                                      & \hbox{if}~ j_1=i_2;\\
    \mathfrak{h}_0(i_1,j_1-i_2+j_2,[0;i_2-j_1+k_2]),                      & \hbox{if}~ j_1>i_2 &\hbox{and}~~ 1\leqslant i_2-j_1+k_2;\\
    \mathfrak{h}_0(\boldsymbol{0}),                                       & \hbox{if}~ j_1>i_2 &\hbox{and}~~ 0=i_2-j_1+k_2;\\
    \mathfrak{h}_0(\boldsymbol{0}),                                       & \hbox{if}~ j_1>i_2 &\hbox{and}~~ i_2-j_1+k_2<0
  \end{array}
\right.=\\
%
  =&
\left\{
  \begin{array}{lll}
    \boldsymbol{0},                          & \hbox{if}~ j_1<i_2 &\hbox{and}~~ j_1-i_2+k_1\leqslant 0; \\
    (i_1-j_1+i_2,j_2,[0;j_1-i_2+k_1-1]),     & \hbox{if}~ j_1<i_2 &\hbox{and}~~ k_2\geqslant j_1-i_2+k_1\geqslant 1; \\
    (i_1-j_1+i_2,j_2,[0;k_2-1]),             & \hbox{if}~ j_1<i_2 &\hbox{and}~~ k_2< j_1-i_2+k_1\geqslant 1; \\
    (i_1,j_2,[0;k_2-1]),                     & \hbox{if}~ j_1=i_2;\\
    (i_1,j_1-i_2+j_2,[0;i_2-j_1+k_2-1]),     & \hbox{if}~ j_1>i_2 &\hbox{and}~~ 1\leqslant i_2-j_1+k_2;\\
    \boldsymbol{0},                          & \hbox{if}~ j_1>i_2 &\hbox{and}~~ i_2-j_1+k_2\leqslant 0.
  \end{array}
\right.
\end{align*}

Next observe that by induction we obtain that
\begin{equation*}
  \mathfrak{h}_p=\underbrace{\mathfrak{h}_0\circ\cdots\circ \mathfrak{h}_0}_{p+1\hbox{-\tiny{times}}}=\mathfrak{h}_0^{p+1}.
\end{equation*}
for any $p=1,\ldots,n-1$.

Simple verifications show that the homomorphism $\mathfrak{h}_p\colon \boldsymbol{B}_{\omega}^{\mathscr{F}_n}\to \boldsymbol{B}_{\omega}^{\mathscr{F}_n}$ maps the semigroup $\boldsymbol{B}_{\omega}^{\mathscr{F}_n}$ onto its subsemigroup $\boldsymbol{B}_{\omega}^{\mathscr{F}_{n-p-1}}$.
\end{proof}

\begin{proposition}\label{proposition-2.11}
For any positive integer $n$ every congruence on the semigroup $\mathscr{I}_\omega^{n}(\overrightarrow{\mathrm{conv}})$ is Rees.
\end{proposition}

\begin{proof}
First we observe that since the semigroup $\mathscr{I}_\omega^{n}(\overrightarrow{\mathrm{conv}})$ has the zero $\boldsymbol{0}$ the identity congruence on $\mathscr{I}_\omega^{n}(\overrightarrow{\mathrm{conv}})$ is Rees, and it is obvious that the universal congruence on $\mathscr{I}_\omega^{n}(\overrightarrow{\mathrm{conv}})$  is Rees, too.

By induction we shall show the following: if $\mathfrak{C}$ is a congruence $\mathscr{I}_\omega^{n}(\overrightarrow{\mathrm{conv}})$ such that for some $k\leqslant n$ there exist two distinct $\mathfrak{C}$-equivalent elements $\alpha,\beta\in \mathscr{I}_\omega^{k}(\overrightarrow{\mathrm{conv}})$ with $\max\{\operatorname{rank}\alpha,\operatorname{rank}\beta\}=k$, then all elements of subsemigroup $\mathscr{I}_\omega^{k}(\overrightarrow{\mathrm{conv}})$ are equivalent.

In the case when $k=1$ then it is obvious that the semigroup $\mathscr{I}_\omega^{1}(\overrightarrow{\mathrm{conv}})$ is isomorphic to the semigroup $\mathscr{I}_\omega^{1}$ which is isomorphic to the semigroup of $\omega{\times}\omega$-matrix units $\mathscr{B}_{\omega}$. Since the semigroup $\mathscr{B}_{\omega}$ of $\omega{\times}\omega$-matrix units is congruence-free (see \cite[Corollary~3]{Gutik-Pavlyk=2005}), the statement that any two distinct elements of the semigroup $\mathscr{I}_\omega^{1}(\overrightarrow{\mathrm{conv}})$ are $\mathfrak{C}$-equivalent implies that all elements of $\mathscr{I}_\omega^{1}(\overrightarrow{\mathrm{conv}})$ are $\mathfrak{C}$-equivalent. Hence the initial step of induction holds.

Next we shall show the step of induction: if $\mathfrak{C}$ is a congruence $\mathscr{I}_\omega^{n}(\overrightarrow{\mathrm{conv}})$ such that there exist two distinct $\mathfrak{C}$-equivalent elements $\alpha,\beta\in \mathscr{I}_\omega^{k+1}(\overrightarrow{\mathrm{conv}})$ with $\max\{\operatorname{rank}\alpha,\operatorname{rank}\beta\}=k+1$, then the statement that all elements of the subsemigroup $\mathscr{I}_\omega^{k}(\overrightarrow{\mathrm{conv}})$ are $\mathfrak{C}$-equivalent implies that all elements of the subsemigroup $\mathscr{I}_\omega^{k+1}(\overrightarrow{\mathrm{conv}})$ are $\mathfrak{C}$-equivalent, as well.

\smallskip

Next we consider all possible cases.

\smallskip

\textbf{(I)}. Suppose that
$\alpha=
\left(
  \begin{smallmatrix}
    a & a+1 & \cdots & a+k \\
    b & b+1 & \cdots & b+k \\
  \end{smallmatrix}
\right)
$, $\beta=\boldsymbol{0}$ and $\alpha\mathfrak{C}\beta$. Since $C$ is a congruence on $\mathscr{I}_\omega^{n}(\overrightarrow{\mathrm{conv}})$, for any element  $\gamma=
\left(
  \begin{smallmatrix}
    c & c+1 & \cdots & c+k_1 \\
    d & d+1 & \cdots & d+k_1 \\
  \end{smallmatrix}
\right)
$ of the subsemigroup $\mathscr{I}_\omega^{k+1}(\overrightarrow{\mathrm{conv}})$, where $k_1\leqslant k+1$, we have that
\begin{equation*}
  \gamma=
\left(
  \begin{smallmatrix}
    c & c+1 & \cdots & c+k_1 \\
    a & a+1 & \cdots & a+k_1 \\
  \end{smallmatrix}
\right)  \cdot\alpha\cdot
\left(
  \begin{smallmatrix}
    b & b+1 & \cdots & b+k_1 \\
    d & d+1 & \cdots & d+k_1 \\
  \end{smallmatrix}
\right)
\end{equation*}
is $\mathfrak{C}$-equivalent to
\begin{equation*}
\left(
  \begin{smallmatrix}
    c & c+1 & \cdots & c+k_1 \\
    a & a+1 & \cdots & a+k_1 \\
  \end{smallmatrix}
\right)  \cdot\boldsymbol{0}\cdot
\left(
  \begin{smallmatrix}
    b & b+1 & \cdots & b+k_1 \\
    d & d+1 & \cdots & d+k_1 \\
  \end{smallmatrix}
\right)=\boldsymbol{0},
\end{equation*}
and hence $\gamma\mathfrak{C}\boldsymbol{0}$.

\smallskip

\textbf{(II)}. Suppose that
$\alpha=
\left(
  \begin{smallmatrix}
    a & a+1 & \cdots & a+k \\
    a & a+1 & \cdots & a+k \\
  \end{smallmatrix}
\right)
$ and $\beta=
\left(
  \begin{smallmatrix}
    b & b+1 & \cdots & b+k_1 \\
    b & b+1 & \cdots & b+k_1 \\
  \end{smallmatrix}
\right)$ are non-zero $\mathfrak{C}$-equivalent idempotents of the subsemigroup $\mathscr{I}_\omega^{k}(\overrightarrow{\mathrm{conv}})$ such that $k_1\leqslant k$ and $\beta\preccurlyeq\alpha$. In this case we have that $[b; b+k_1]\subseteq [a;a+k]$. We put
\begin{equation*}
  \varepsilon=
  \left\{
    \begin{array}{ll}
\left(
  \begin{smallmatrix}
    a+1 & \cdots & a+k \\
    a+1 & \cdots & a+k \\
  \end{smallmatrix}
\right)
, & \hbox{if~} a=b;\\
\left(
  \begin{smallmatrix}
    a & \cdots & a+k-1 \\
    a & \cdots & a+k-1 \\
  \end{smallmatrix}
\right)
, & \hbox{if~} a+k=b+k_1
    \end{array}
  \right.
\end{equation*}
and $\gamma=
\left(
  \begin{smallmatrix}
    a   & a+1 & \cdots & a+k \\
    a+1 & a+2 & \cdots & a+k+1 \\
  \end{smallmatrix}
\right)$
if $a<b$ and $b+k_1<a+k$.

In the case when either $a=b$ or $a+k=b+k_1$ we obtain that $\varepsilon\alpha$ and $\varepsilon\beta$ are distinct $\mathfrak{C}$-equivalent idempotents of the subsemigroup $\mathscr{I}_\omega^{k-1}(\overrightarrow{\mathrm{conv}})$ and hence by the assumption of induction all elements of $\mathscr{I}_\omega^{k-1}(\overrightarrow{\mathrm{conv}})$ are $\mathfrak{C}$-equivalent.

In the case when $a<b$ and $b+k_1<a+k$ we obtain that $\gamma\alpha\gamma^{-1}$ and $\gamma\beta\gamma^{-1}$ are distinct $\mathfrak{C}$-equivalent idempotents of the subsemigroup $\mathscr{I}_\omega^{k}(\overrightarrow{\mathrm{conv}})$, because they have distinct rank $\leqslant k$. Hence by the assumption of induction all elements of $\mathscr{I}_\omega^{k}(\overrightarrow{\mathrm{conv}})$ are $\mathfrak{C}$-equivalent.

In both above cases we get that $\alpha\mathfrak{C}\boldsymbol{0}$, which implies that case \textbf{(I)} holds.

\smallskip

\textbf{(III)}. Suppose that
$\alpha$ and $\beta$ are distinct incomparable non-zero $\mathfrak{C}$-equivalent idempotents of the subsemigroup $\mathscr{I}_\omega^{k}(\overrightarrow{\mathrm{conv}})$ of $\mathscr{I}_\omega^{n}(\overrightarrow{\mathrm{conv}})$ such that $\operatorname{rank}\alpha=k+1$. Then $\alpha=\alpha\alpha\mathfrak{C}\alpha\beta$ and $\alpha\beta\preccurlyeq\alpha$ which implies that either case \textbf{(II)} or case \textbf{(I)} holds.

\smallskip

\textbf{(IV)}. Suppose that
$\alpha$ and $\beta$ are distinct non-zero $\mathfrak{C}$-equivalent elements of the subsemigroup $\mathscr{I}_\omega^{k}(\overrightarrow{\mathrm{conv}})$ of $\mathscr{I}_\omega^{n}(\overrightarrow{\mathrm{conv}})$ such that $\operatorname{rank}\alpha=k+1$. Then at least one of the following conditions holds $\alpha\alpha^{-1}\neq\beta\beta^{-1}$ or $\alpha^{-1}\alpha\neq\beta^{-1}\beta$, because by Proposition~\ref{proposition-2.1}\eqref{proposition-2.1(8)} and Theorem~\ref{theorem-2.5}  all $\mathscr{H}$-classes in $\mathscr{I}_\omega^{n}(\overrightarrow{\mathrm{conv}})$ are singletons. By Proposition~2.3.4(1) of \cite{Lawson-1998}, $\alpha\alpha^{-1}\mathfrak{C}\beta\beta^{-1}$ and $\alpha^{-1}\alpha\mathfrak{C}\beta^{-1}\beta$, and hence at least one of cases \textbf{(II)} or \textbf{(III)} holds.
\end{proof}

Theorem~\ref{theorem-2.5} and Proposition~\ref{proposition-2.11} imply the description of all congruences on the semigroup $\boldsymbol{B}_{\omega}^{\mathscr{F}_n}$:

\begin{theorem}\label{theorem-2.12}
For an arbitrary $n\in\omega$ the semigroup $\boldsymbol{B}_{\omega}^{\mathscr{F}_n}$ admits only Rees congruences.
\end{theorem}

\begin{theorem}\label{theorem-2.13}
Let $n$ be a non-negative integer and $S$ be a semigroup. For any homomorphism $\mathfrak{h}\colon \boldsymbol{B}_{\omega}^{\mathscr{F}_n}\to S$ the image $\mathfrak{h}(\boldsymbol{B}_{\omega}^{\mathscr{F}_n})$ is either isomorphic to $\boldsymbol{B}_{\omega}^{\mathscr{F}_k}$ for some $k=0,1,\ldots,n$, or is a singleton.
\end{theorem}

\begin{proof}
By Theorem~\ref{theorem-2.12} the homomorphism $\mathfrak{h}$ generates the Rees congruence $\mathfrak{C}_\mathfrak{h}$ on the semigroup $\boldsymbol{B}_{\omega}^{\mathscr{F}_n}$. By Proposition~\ref{proposition-2.1}\eqref{proposition-2.1(9)} the following ideal series
\begin{equation*}
  \{\boldsymbol{0}\}\subsetneqq \boldsymbol{B}_{\omega}^{\mathscr{F}_0}\subsetneqq \boldsymbol{B}_{\omega}^{\mathscr{F}_1}\subsetneqq \cdots \subsetneqq \boldsymbol{B}_{\omega}^{\mathscr{F}_{n-1}}\subsetneqq \boldsymbol{B}_{\omega}^{\mathscr{F}_n}
\end{equation*}
is maximal in $\boldsymbol{B}_{\omega}^{\mathscr{F}_n}$, i.e., if $\mathscr{J}$ is an ideal of $\boldsymbol{B}_{\omega}^{\mathscr{F}_n}$ then either $\mathscr{J}=\{\boldsymbol{0}\}$ or $\mathscr{J}=\boldsymbol{B}_{\omega}^{\mathscr{F}_m}$ for some $m=0,1,\ldots,n$.

It is obvious that if $\mathscr{J}=\{\boldsymbol{0}\}$ then the Rees congruence $\mathfrak{C}_\mathscr{J}$ generates the injective homomorphism $\mathfrak{h}_{\mathfrak{C}_\mathscr{J}}$, and hence the image $\mathfrak{h}_{\mathfrak{C}_\mathscr{J}}(\boldsymbol{B}_{\omega}^{\mathscr{F}_n})$ is isomorphic to the semigroup $\boldsymbol{B}_{\omega}^{\mathscr{F}_n}$. Similar in the case when $\mathscr{J}=\boldsymbol{B}_{\omega}^{\mathscr{F}_n}$ we have that the image $\mathfrak{h}_{\mathfrak{C}_\mathscr{J}}(\boldsymbol{B}_{\omega}^{\mathscr{F}_n})$ is a singleton.

Suppose that $\mathscr{J}=\boldsymbol{B}_{\omega}^{\mathscr{F}_m}$ for some $m=0,1,\ldots,n-1$. Then the Rees congruence $\mathfrak{C}_\mathscr{J}$ generates the natural homomorphism $\mathfrak{h}\colon \boldsymbol{B}_{\omega}^{\mathscr{F}_n}\to \boldsymbol{B}_{\omega}^{\mathscr{F}_n}/\mathscr{J}$. It is obvious that $\alpha\mathfrak{C}_\mathscr{J}\beta$ if and only if $\mathfrak{h}_m(\alpha)=\mathfrak{h}_m(\beta)$ for $\alpha,\beta\in \boldsymbol{B}_{\omega}^{\mathscr{F}_n}$ where $\mathfrak{h}_m\colon \boldsymbol{B}_{\omega}^{\mathscr{F}_n}\to \boldsymbol{B}_{\omega}^{\mathscr{F}_n}$ is the homomorphism defined in Proposition~\ref{proposition-2.10}.  Then by Proposition~\ref{proposition-2.10} the image $\mathfrak{h}(\boldsymbol{B}_{\omega}^{\mathscr{F}_n})$ is isomorphic to the semigroup $\boldsymbol{B}_{\omega}^{\mathscr{F}_{n-m-1}}$.
\end{proof}


\section{On topologizations and closure of the semigroup $\boldsymbol{B}_{\omega}^{\mathscr{F}_n}$}

In this section we establish topologizations of the semigroup $\boldsymbol{B}_{\omega}^{\mathscr{F}_n}$ and its compact-like shift-continuous topologies.

\begin{theorem}\label{theorem-3.1}
Let $n$ be a non-negative integer. Then for any shift-continuous $T_1$-topology $\tau$ on the semigroup $\boldsymbol{B}_{\omega}^{\mathscr{F}_n}$ every non-zero element of $\boldsymbol{B}_{\omega}^{\mathscr{F}_n}$ is an isolated point of $(\boldsymbol{B}_{\omega}^{\mathscr{F}_n},\tau)$ and hence every  subset in $(\boldsymbol{B}_{\omega}^{\mathscr{F}_n},\tau)$ which contains zero is closed. Moreover, for any non-zero element $(i,j,[0;k])$ of  $\boldsymbol{B}_{\omega}^{\mathscr{F}_n}$ the set ${\uparrow_{\preccurlyeq}}(i,j,[0;k])$ is open-and-closed in $(\boldsymbol{B}_{\omega}^{\mathscr{F}_n},\tau)$.
\end{theorem}

\begin{proof}
Fix an arbitrary non-zero element $(i,j,[0;k])$ of the semigroup $\boldsymbol{B}_{\omega}^{\mathscr{F}_n}$, $i,j\in\omega$, $k\in\{0,\ldots,n\}$. Proposition~7 of \cite{Gutik-Lawson-Repov=2009} and Proposition~\ref{proposition-2.10} imply there exists an open neighbourhood $U_{(i,j,[0;k])}$ of the point $(i,j,[0;k])$ in $(\boldsymbol{B}_{\omega}^{\mathscr{F}_n},\tau)$ such that
\begin{itemize}
  \item $U_{(i,j,[0;k])}\subseteq \boldsymbol{B}_{\omega}^{\mathscr{F}_n}\setminus \boldsymbol{B}_{\omega}^{\mathscr{F}_{k-1}}$ and $(i,j,[0;k])$ is an isolated point in $\boldsymbol{B}_{\omega}^{\mathscr{F}_k}$ if $k\in\{1,\ldots,n\}$; \quad and
  \item $U_{(i,j,[0;k])}\subseteq \boldsymbol{B}_{\omega}^{\mathscr{F}_n}\setminus\{\boldsymbol{0}\}$ and $(i,j,[0;k])$ is an isolated point in $\boldsymbol{B}_{\omega}^{\mathscr{F}_0}$ if $k=0$.
\end{itemize}
By separate continuity of the semigroup operation in $(\boldsymbol{B}_{\omega}^{\mathscr{F}_n},\tau)$ there exists an open neighbourhood $V_{(i,j,[0;k])}$ of $(i,j,[0;k])$ such that $V_{(i,j,[0;k])}\subseteq U_{(i,j,[0;k])}$ and
\begin{equation*}
(i,i,[0;k])\cdot V_{(i,j,[0;k])}\cdot (j,j,[0;k])\subseteq U_{(i,j,[0;k])}.
\end{equation*}
We claim that $V_{(i,j,[0;k])}\subseteq {\uparrow_{\preccurlyeq}}(i,j,[0;k])$. Suppose to the contrary that there exists $(i_1,j_1,[0;k_1])\in V_{(i,j,[0;k])}\setminus {\uparrow_{\preccurlyeq}}(i,j,[0;k])$. Then by Lemma~1.4.6(4) of \cite{Lawson-1998} we have that
\begin{equation*}
(i,i,[0;k])\cdot(i_1,j_1,[0;k_1])\cdot (j,j,[0;k])\neq (i,j,[0;k]).
\end{equation*}
Since $\boldsymbol{B}_{\omega}^{\mathscr{F}_k}$ is an ideal of $\boldsymbol{B}_{\omega}^{\mathscr{F}_n}$ the above inequality implies that
\begin{equation*}
(i,i,[0;k])\cdot V_{(i,j,[0;k])}\cdot (j,j,[0;k])\nsubseteq U_{(i,j,[0;k])},
\end{equation*}
a contradiction. Hence $V_{(i,j,[0;k])}\subseteq {\uparrow_{\preccurlyeq}}(i,j,[0;k])$. By Lemma~\ref{lemma-2.2} the set ${\uparrow_{\preccurlyeq}}(i,j,[0;k])$ is finite which implies that $(i,j,[0;k])$ is an isolated point of $(\boldsymbol{B}_{\omega}^{\mathscr{F}_n},\tau)$, because $(\boldsymbol{B}_{\omega}^{\mathscr{F}_n},\tau)$ is a  $T_1$-space.

The last statement follows from the equality
\begin{equation*}
  {\uparrow_{\preccurlyeq}}(i,j,[0;k])=\left\{(a,b,[0;p])\in \boldsymbol{B}_{\omega}^{\mathscr{F}_n}\colon (i,i,[0;k])\cdot(a,b,[0;p])=(i,j,[0;k])\right\}
\end{equation*}
and the assumption that $\tau$ is a shift-continuous $T_1$-topology on the semigroup $\boldsymbol{B}_{\omega}^{\mathscr{F}_n}$.
\end{proof}

Recall \cite{Engelking-1989} a topological space $X$ is called:
\begin{itemize}
  \item \emph{scattered} if $X$ contains no non-empty subset which is dense-in-itself;
  \item \emph{$0$-dimensional} if $X$ has a base which consists of open-and-closed subsets;
  \item \emph{collectionwise normal} if for every discrete family $\{F_i\}_{i\in \mathscr{S}}$ of closed subsets of $X$ there exists a pairwise disjoint family of open sets $\{U_i\}_{i\in \mathscr{S}}$, such that $F_i\subseteq U_i$ for all $i\in \mathscr{S}$.
\end{itemize}

\begin{corollary}\label{corollary-3.2}
Let $n$ be a non-negative integer. Then for any shift-continuous $T_1$-topology $\tau$ on the semigroup $\boldsymbol{B}_{\omega}^{\mathscr{F}_n}$ the space  $(\boldsymbol{B}_{\omega}^{\mathscr{F}_n},\tau)$ is scattered, $0$-dimensional and collectionwise normal.
\end{corollary}

\begin{proof}
Theorem~\ref{theorem-3.1} implies that $(\boldsymbol{B}_{\omega}^{\mathscr{F}_n},\tau)$ is a scattered, $0$-dimensional space.

Let $\{F_s\}_{s\in\mathscr{S}}$ be a discrete family  of closed subsets of $\left(\boldsymbol{B}_{\omega}^{\mathscr{F}_n},\tau\right)$. By Theorem~\ref{theorem-3.1} every non-zero element of $\boldsymbol{B}_{\omega}^{\mathscr{F}_n}$ is an isolated point of $(\boldsymbol{B}_{\omega}^{\mathscr{F}_n},\tau)$. In the case when every element of the family $\{F_s\}_{s\in\mathscr{S}}$ does not contain the zero $\boldsymbol{0}$ of $\boldsymbol{B}_{\omega}^{\mathscr{F}_n}$  by Theorem~5.1.17 from \cite{Engelking-1989} the space $\left(\boldsymbol{B}_{\omega}^{\mathscr{F}_n},\tau\right)$ is collectionwise normal. Suppose that $\boldsymbol{0}\in F_{s_0}$ for some $s_0\in \mathscr{S}$. Let $U(\boldsymbol{0})$ be an open neighbourhood of the zero $\boldsymbol{0}$ of $\boldsymbol{B}_{\omega}^{\mathscr{F}_n}$ which intersects at more one element of the family $\{F_s\}_{s\in\mathscr{S}}$. Put $U_{s_0}=U(0)\cup F_{s_0}$ and $U_s=F_s$ for all $s\in \mathscr{S}\setminus\{s_o\}$. Then $U_s\cap U_t=\varnothing$ for all distinct $s,t\in\mathscr{S}$ and hence by Theorem~5.1.17 from \cite{Engelking-1989} the space $\left(\boldsymbol{B}_{\omega}^{\mathscr{F}_n},\tau\right)$ is collectionwise normal.
\end{proof}

\begin{example}\label{example-3.3}
Let $n$ be a non-negative integer.  We define a topology $\tau_{\mathrm{Ac}}$  on the semigroup $\boldsymbol{B}_{\omega}^{\mathscr{F}_n}$ in the following way. All non-zero elements of the semigroup $\boldsymbol{B}_{\omega}^{\mathscr{F}_n}$ are isolated points of $(\boldsymbol{B}_{\omega}^{\mathscr{F}_n},\tau_{\mathrm{Ac}})$ and the family $\mathscr{B}_{\mathrm{Ac}}(\boldsymbol{0})=\left\{A\subseteq \boldsymbol{B}_{\omega}^{\mathscr{F}_n}\colon \boldsymbol{0}\in A \; \hbox{and} \; \boldsymbol{B}_{\omega}^{\mathscr{F}_n}\setminus A \; \hbox{is finite}\right\}$ determines the base of the topology $\tau_{\mathrm{Ac}}$ at the point $\boldsymbol{0}$.

It is obvious that the topological space $\left(\boldsymbol{B}_{\omega}^{\mathscr{F}_n},\tau_{\mathrm{Ac}}\right)$ is homeomorphic to the Alexandroff one-point compactification of the discrete infinite countable space, and hence $\left(\boldsymbol{B}_{\omega}^{\mathscr{F}_n},\tau_{\mathrm{Ac}}\right)$ is a Hausdorff compact space. Then the space $\left(\boldsymbol{B}_{\omega}^{\mathscr{F}_n},\tau_{\mathrm{Ac}}\right)$ is normal and since it has a countable base, by the Urysohn Metrization Theorem (see \cite[Theorem~4.2.9]{Engelking-1989}) the space $\left(\boldsymbol{B}_{\omega}^{\mathscr{F}_n},\tau_{\mathrm{Ac}}\right)$ is metrizable.

Next we shall show that $\left(\boldsymbol{B}_{\omega}^{\mathscr{F}_n},\tau_{\mathrm{Ac}}\right)$ is a semitopological semigroup. Let $\alpha$ and $\beta$ be non-zero elements of the semigroup $\boldsymbol{B}_{\omega}^{\mathscr{F}_n}$. Since $\alpha$ and $\beta$ are isolated points in $\left(\boldsymbol{B}_{\omega}^{\mathscr{F}_n},\tau_{\mathrm{Ac}}\right)$, it is sufficient to show how to find for a fixed open neighbourhood $U_{\boldsymbol{0}}$ open neighbourhoods $V_{\boldsymbol{0}}$ and $W_{\boldsymbol{0}}$ of the zero $\boldsymbol{0}$  in $\left(\boldsymbol{B}_{\omega}^{\mathscr{F}_n},\tau_{\mathrm{Ac}}\right)$ such that
\begin{equation*}
  V_{\boldsymbol{0}}\cdot \alpha\subseteq U_{\boldsymbol{0}} \qquad \hbox{and} \qquad \beta\cdot W_{\boldsymbol{0}}\subseteq U_{\boldsymbol{0}}.
\end{equation*}
Since the space $\left(\boldsymbol{B}_{\omega}^{\mathscr{F}_n},\tau_{\mathrm{Ac}}\right)$ is compact, any open neighbourhood $U_{\boldsymbol{0}}$ of the zero $\boldsymbol{0}$ is cofinite subset in $\boldsymbol{B}_{\omega}^{\mathscr{F}_n}$. By Lemma~\ref{lemma-2.3},
\begin{equation*}
  V_{\boldsymbol{0}}=\{\gamma\in U_{\boldsymbol{0}}\colon \gamma\cdot\alpha\in U_{\boldsymbol{0}}\} \qquad \hbox{and} \qquad W_{\boldsymbol{0}}=\{\gamma\in U_{\boldsymbol{0}}\colon\beta\cdot\gamma\in U_{\boldsymbol{0}}\}
\end{equation*}
are cofinite subsets of $U_{\boldsymbol{0}}$ and hence by the definition of the topology $\tau_{\mathrm{Ac}}$ the sets $V_{\boldsymbol{0}}$ and $W_{\boldsymbol{0}}$ are required open neighbourhoods of the zero $\boldsymbol{0}$  in $\left(\boldsymbol{B}_{\omega}^{\mathscr{F}_n},\tau_{\mathrm{Ac}}\right)$.

Since all non-zero elements of the semigroup $\boldsymbol{B}_{\omega}^{\mathscr{F}_n}$ are isolated point in $\left(\boldsymbol{B}_{\omega}^{\mathscr{F}_n},\tau_{\mathrm{Ac}}\right)$ and every open neighbourhood $U_{\boldsymbol{0}}$ of the zero in $\left(\boldsymbol{B}_{\omega}^{\mathscr{F}_n},\tau_{\mathrm{Ac}}\right)$ has the finite complement in $\boldsymbol{B}_{\omega}^{\mathscr{F}_n}$, the inversion is continuous in $\left(\boldsymbol{B}_{\omega}^{\mathscr{F}_n},\tau_{\mathrm{Ac}}\right)$.
\end{example}

The following theorem describes all compact-like shift-continuous $T_1$-topologies on the semigroup $\boldsymbol{B}_{\omega}^{\mathscr{F}_n}$.

\begin{theorem}\label{theorem-3.4}
Let $n$ be a non-negative integer. Then for any shift-continuous $T_1$-topology $\tau$ on the semigroup $\boldsymbol{B}_{\omega}^{\mathscr{F}_n}$ the following conditions are equivalent:
\begin{enumerate}
  \item\label{theorem-3.3(1)} $(\boldsymbol{B}_{\omega}^{\mathscr{F}_n},\tau)$ is a compact semitopological semigroup;
  \item\label{theorem-3.3(2)} $(\boldsymbol{B}_{\omega}^{\mathscr{F}_n},\tau)$ is topologically isomorphic to $\left(\boldsymbol{B}_{\omega}^{\mathscr{F}_n},\tau_{\mathrm{Ac}}\right)$;
  \item\label{theorem-3.3(3)} $(\boldsymbol{B}_{\omega}^{\mathscr{F}_n},\tau)$ is a compact semitopological semigroup with continuous inversion;
  \item\label{theorem-3.3(4)} $(\boldsymbol{B}_{\omega}^{\mathscr{F}_n},\tau)$ is an $\omega_{\mathfrak{d}}$-compact space.
\end{enumerate}
\end{theorem}

\begin{proof}
Implications \eqref{theorem-3.3(1)}$\Rightarrow$\eqref{theorem-3.3(4)},  \eqref{theorem-3.3(2)}$\Rightarrow$\eqref{theorem-3.3(1)}, \eqref{theorem-3.3(2)}$\Rightarrow$\eqref{theorem-3.3(3)} and \eqref{theorem-3.3(3)}$\Rightarrow$\eqref{theorem-3.3(1)} are obvious. Since by Theorem~\ref{theorem-3.1} every non-zero element of the semigroup $\boldsymbol{B}_{\omega}^{\mathscr{F}_n}$ is an isolated point in $(\boldsymbol{B}_{\omega}^{\mathscr{F}_n},\tau)$, statement \eqref{theorem-3.3(1)} implies \eqref{theorem-3.3(2)}.

\eqref{theorem-3.3(4)}$\Rightarrow$\eqref{theorem-3.3(1)} Suppose there exists a shift-continuous $T_1$-topology $\tau$ on the semigroup $\boldsymbol{B}_{\omega}^{\mathscr{F}_n}$ such that $(\boldsymbol{B}_{\omega}^{\mathscr{F}_n},\tau)$ is an $\omega_{\mathfrak{d}}$-compact non-compact space. Then there exists an open cover $\mathscr{U}=\{U_s\}$ of $(\boldsymbol{B}_{\omega}^{\mathscr{F}_n},\tau)$ which has no a finite subcover. Let $U_{s_0}\in\mathscr{U}$ be such that $U_{s_0}\ni\boldsymbol{0}$. Then $\boldsymbol{B}_{\omega}^{\mathscr{F}_n}\setminus U_{s_0}$ is an infinite countable subset of isolated points of $(\boldsymbol{B}_{\omega}^{\mathscr{F}_n},\tau)$. We enumerate the set $\boldsymbol{B}_{\omega}^{\mathscr{F}_n}\setminus U_{s_0}$ by positive integers, i.e., $\boldsymbol{B}_{\omega}^{\mathscr{F}_n}\setminus U_{s_0}=\{\alpha_i\colon i\in \mathbb{N}\}$. Next we define a map $f\colon (\boldsymbol{B}_{\omega}^{\mathscr{F}_n},\tau)\to \omega_{\mathfrak{d}}$ by the formula
\begin{equation*}
  f(\alpha)=
  \left\{
    \begin{array}{ll}
      0, & \hbox{if~} \alpha\in U_{s_0};\\
      i, & \hbox{if~} \alpha=\alpha_i \hbox{~for some~} i\in\mathbb{N}.
    \end{array}
  \right.
\end{equation*}
By Theorem~\ref{theorem-3.1} the set $U_{s_0}$ is open-and-closed in $(\boldsymbol{B}_{\omega}^{\mathscr{F}_n},\tau)$, and hence so defined map $f$ is continuous. But the image $f(\boldsymbol{B}_{\omega}^{\mathscr{F}_n})$ is not a compact subset of $\omega_{\mathfrak{d}}$, a contradiction. The obtained contradiction implies the implication \eqref{theorem-3.3(4)}$\Rightarrow$\eqref{theorem-3.3(1)}.
\end{proof}

The following proposition states that the semigroup $\boldsymbol{B}_{\omega}^{\mathscr{F}_n}$ has a similar closure in a $T_1$-semitopological semigoup as the bicyclic monoid (see \cite{Bertman-West=1976} and \cite{Eberhart-Selden=1969}), the $\lambda$-polycyclic monoid \cite{Bardyla-Gutik=2016}, graph inverse semigroups \cite{Bardyla=2020, Mesyan-Mitchell-Morayne-Peresse=2016}, McAlister semigroups \cite{Bardyla=2021??}, locally compact semitopological $0$-bisimple inverse $\omega$-semigroups with a compact maximal subgroup \cite{Gutik=2018}, and other discrete semigroups of bijective partial transformations \cite{Chuchman-Gutik=2010, Chuchman-Gutik=2011, Gutik-P.Khylynskyi=2021, Gutik-Mokrytskyi=2020, Gutik-Pozdnyakova=2014, Gutik-Repovs=2011, Gutik-Repovs=2012, Gutik-Savchuk=2017, Gutik-Savchuk=2019}.

\begin{proposition}\label{proposition-3.5}
Let $n$ be a non-negative integer.
If $S$ is a $T_1$-semitopological semigroup which contains $\boldsymbol{B}_{\omega}^{\mathscr{F}_n}$ as a dense proper subsemigroup then $I=(S\setminus \boldsymbol{B}_{\omega}^{\mathscr{F}_n})\cup\{\boldsymbol{0}\}$ is an ideal of $S$.
\end{proposition}

\begin{proof}
Fix an arbitrary element $\nu\in I$. If $\chi\cdot \nu=\zeta\notin I$ for some $\chi\in \boldsymbol{B}_{\omega}^{\mathscr{F}_n}$ then there exists an open neighbourhood $U(\nu)$ of the point $\nu$ in the space $S$ such that $\{\chi\}\cdot U(\nu)=\{\zeta\}\subset \boldsymbol{B}_{\omega}^{\mathscr{F}_n}\setminus\{\boldsymbol{0}\}$. By Lemma~\ref{lemma-2.4} the open neighbourhood $U(\nu)$ should contain finitely many elements of the semigroup $\boldsymbol{B}_{\omega}^{\mathscr{F}_n}$ which contradicts our assumption. Hence $\chi\cdot \nu\in I$ for all $\chi\in \boldsymbol{B}_{\omega}^{\mathscr{F}_n}$ and $\nu\in I$. The proof of the statement that $\nu\cdot \chi\in I$ for all $\chi\in \boldsymbol{B}_{\omega}^{\mathscr{F}_n}$ and $\nu\in I$ is similar.

Suppose to the contrary that $\chi\cdot \nu=\omega\notin I$ for some $\chi,\nu\in I$. Then $\omega\in \boldsymbol{B}_{\omega}^{\mathscr{F}_n}$ and the separate continuity of the semigroup operation in $S$ yields open neighbourhoods $U(\chi)$ and $U(\nu)$ of the points $\chi$ and $\nu$ in the space $S$, respectively, such that $\{\chi\}\cdot U(\nu)=\{\omega\}$ and $U(\chi)\cdot \{\nu\}=\{\omega\}$. Since both neighbourhoods $U(\chi)$ and $U(\nu)$ contain infinitely many elements of the semigroup $\boldsymbol{B}_{\omega}^{\mathscr{F}_n}$,  equalities $\{\chi\}\cdot U(\nu)=\{\omega\}$ and $U(\chi)\cdot \{\nu\}=\{\omega\}$ do not hold, because $\{\chi\}\cdot \left(U(\nu)\cap \boldsymbol{B}_{\omega}^{\mathscr{F}_n}\right)\subseteq I$. The obtained contradiction implies that $\chi\cdot \nu\in I$.
\end{proof}

For any $k=0,1,\ldots,n+1$ we denote
\begin{equation*}
  D_k=\left\{\alpha\in\mathscr{I}_\omega^{n+1}(\overrightarrow{\mathrm{conv}})\colon \operatorname{rank}\alpha=k \right\}.
\end{equation*}
We observe that by Proposition~\ref{proposition-2.1}\eqref{proposition-2.1(9)} and Theorem~\ref{theorem-2.5}, $\boldsymbol{D}=\{D_k\colon k=0,1,\ldots,n+1\}$ is the family of all $\mathscr{D}$-classed of the semigroup $\mathscr{I}_\omega^{n+1}(\overrightarrow{\mathrm{conv}})$.

The following proposition describes the remainder of the semigroup $\boldsymbol{B}_{\omega}^{\mathscr{F}_n}$ in a semitopological semigroup.

\begin{proposition}\label{proposition-3.6}
Let $n$ be a non-negative integer.
If $S$ is a $T_1$-semitopological semigroup which contains $\boldsymbol{B}_{\omega}^{\mathscr{F}_n}$ as a dense proper subsemigroup then $\chi\cdot\chi=\boldsymbol{0}$ for all $\chi\in S\setminus\boldsymbol{B}_{\omega}^{\mathscr{F}_n}$.
\end{proposition}

\begin{proof}
We observe that $\boldsymbol{0}$ is zero of the semigroup $S$ by Lemma 4.4 of \cite{Gutik=2018}.

We shall prove the statement of the proposition for the semigroup $\mathscr{I}_\omega^{n+1}(\overrightarrow{\mathrm{conv}})$ which by Theorem~\ref{theorem-2.5} is isomorphic to the semigroup $\boldsymbol{B}_{\omega}^{\mathscr{F}_n}$.

Fix an arbitrary $\chi\in S\setminus\mathscr{I}_\omega^{n+1}(\overrightarrow{\mathrm{conv}})$ and any open neighbourhood $U(\chi)$ of the point $\chi$ in $S$. Since $\boldsymbol{B}_{\omega}^{\mathscr{F}_n}$ is a dense proper subsemigroup of $S$ the set $U(\chi)\cap (\mathscr{I}_\omega^{n+1}(\overrightarrow{\mathrm{conv}})\setminus \{\boldsymbol{0}\})$ is infinite. Since the family $\boldsymbol{D}$ is finite there exists $i=1,\ldots,n+1$ such that the set $U(\chi)\cap D_i$ is infinite. This and the definition of the semigroup $\mathscr{I}_\omega^{n+1}(\overrightarrow{\mathrm{conv}})$ imply that at least one of the families
\begin{equation*}
  \operatorname{\mathfrak{dom}}D_iU(\chi)=\left\{\operatorname{dom}\alpha\colon \alpha\in U(\chi)\cap D_i \right\} \qquad \hbox{or} \qquad \operatorname{\mathfrak{ran}}D_iU(\chi)=\left\{\operatorname{ran}\alpha\colon \alpha\in U(\chi)\cap D_i \right\}
\end{equation*}
has infinitely many members. Assume that the family $\operatorname{\mathfrak{dom}}D_iU(\chi)$ is infinite. Then the definition of the semigroup operation on $\mathscr{I}_\omega^{n+1}(\overrightarrow{\mathrm{conv}})$ implies that there exist infinitely many $\beta\in U(\chi)\cap \mathscr{I}_\omega^{n+1}(\overrightarrow{\mathrm{conv}})$ such that $\boldsymbol{0}\in\beta\cdot U(\chi)$, and since $S$ is a $T_1$-space we have that $\beta\cdot\chi=\boldsymbol{0}$ for so elements $\beta$. Also, the infiniteness of $\operatorname{\mathfrak{dom}}D_iU(\chi)$ and the semigroup operation of $\mathscr{I}_\omega^{n+1}(\overrightarrow{\mathrm{conv}})$ imply the existence infinitely many $\gamma\in U(\chi)\cap \mathscr{I}_\omega^{n+1}(\overrightarrow{\mathrm{conv}})$ such that $\boldsymbol{0}\in U(\chi)\cdot\gamma$, and since $S$ is a $T_1$-space we have that $\chi\cdot\gamma=\boldsymbol{0}$ for so elements $\gamma$. In the case when the family $\operatorname{\mathfrak{ran}}D_iU(\chi)$ is infinite similarly we obtain that there exist infinitely many $\beta,\gamma\in U(\chi)\cap \mathscr{I}_\omega^{n+1}(\overrightarrow{\mathrm{conv}})$ such that $\beta\cdot\chi=\boldsymbol{0}$ and $\chi\cdot\gamma=\boldsymbol{0}$.

Thus we show that $\boldsymbol{0}\in V(\chi)\cdot\chi$ and  $\boldsymbol{0}\in\chi\cdot V(\chi)$ for any open neighbourhood $V(\chi)$ of the point $\chi$ in $S$. Since $S$ is a $T_1$-space this implies the required equality $\chi\cdot\chi=\boldsymbol{0}$ for all $\chi\in S\setminus\boldsymbol{B}_{\omega}^{\mathscr{F}_n}$.
\end{proof}

Let $\mathfrak{STSG}$ be a class of semitopological semigroups. A semigroup $S\in\mathfrak{STSG}$ is called {\it $H$-closed in} $\mathfrak{STSG}$, if $S$ is a closed subsemigroup of any topological semigroup $T\in\mathfrak{STSG}$ which contains $S$ both as a subsemigroup and as a topological space. $H$-closed
topological semigroups were introduced by Stepp in \cite{Stepp=1969}, and there they were called {\it maximal semigroups}. A semitopological semigroup $S\in\mathfrak{STSG}$ is called {\it absolutely $H$-closed in the class} $\mathfrak{STSG}$, if any continuous homomorphic image of $S$ into $T\in\mathfrak{STSG}$ is $H$-closed in $\mathfrak{STSG}$. An algebraic semigroup $S$ is called:
\begin{itemize}
  \item {\it algebraically complete in} $\mathfrak{STSG}$, if $S$ with any Hausdorff topology $\tau$ such that $(S,\tau)\in\mathfrak{STSG}$ is $H$-closed in $\mathfrak{STSG}_0$;
  \item {\it algebraically $h$-complete in} $\mathfrak{STSG}$, if $S$ with discrete topology ${\tau_\mathfrak{d}}$ is absolutely $H$-closed in $\mathfrak{STSG}$ and $(S,{\tau_\mathfrak{d}})\in\mathfrak{STSG}$.
\end{itemize}
Absolutely $H$-closed topological semigroups and algebraically $h$-complete semigroups were introduced by Stepp in~\cite{Stepp=1975}, and there they were called {\it absolutely maximal} and {\it algebraic maximal}, respectively. Other distinct types of completeness of (semi)topological semigroups were studied by Banakh and Bardyla (see \cite{Bardyla-Bardyla=2019, Bardyla-Bardyla=2019a, Bardyla-Bardyla=2020, Bardyla-Bardyla=2021, Bardyla-Bardyla=2022, Bardyla-Bardyla-Ravsky=2019}).

Proposition~10 of \cite{Gutik-Lawson-Repov=2009} and Proposition~\ref{proposition-2.10} imply the following theorem.

\begin{theorem}\label{theorem-3.7}
For any $n\in\omega$ the semigroup $\boldsymbol{B}_{\omega}^{\mathscr{F}_n}$ is algebraically complete in the class of Hausdorff semitopological inverse semigroups with continuous inversion, and hence in the class of Hausdorff topological inverse semigroups.
\end{theorem}

\begin{theorem}\label{theorem-3.8}
Let $n$ be a non-negative integer.
If $(\boldsymbol{B}_{\omega}^{\mathscr{F}_n},\tau)$ is a Hausdorff topological semigroup with the compact band then $(\boldsymbol{B}_{\omega}^{\mathscr{F}_n},\tau)$ is $H$-closed in the class of Hausdorff topological semigroups.
\end{theorem}

\begin{proof}
Suppose to the contrary that there exists a Hausdorff topological semigroup $T$ which contains $(\boldsymbol{B}_{\omega}^{\mathscr{F}_n},\tau)$ as a non-closed subsemigroup. Since the closure of a subsemigroup of a topological semigroup $S$ is a subsemigroup of $S$ (see \cite[p.~9]{Carruth-Hildebrant-Koch-1983}), without loss of generality we can assume that $\boldsymbol{B}_{\omega}^{\mathscr{F}_n}$ is a dense subsemigroup of $T$ and $T\setminus\boldsymbol{B}_{\omega}^{\mathscr{F}_n}\neq\varnothing$. Let $\chi\in T\setminus\boldsymbol{B}_{\omega}^{\mathscr{F}_n}$. Then $\boldsymbol{0}$ is the zero of the semigroup $T$ by Lemma 4.4 of \cite{Gutik=2018}, and $\chi\cdot\chi=\boldsymbol{0}$ by Proposition~\ref{proposition-3.6}.

Since $\boldsymbol{0}\cdot\chi=\chi\cdot\boldsymbol{0}=\boldsymbol{0}$ and $T$ is a Hausdorff topological semigroup, for any disjoint open neighbourhoods $U(\chi)$ and $U(\boldsymbol{0})$ of $\chi$ and $\boldsymbol{0}$ in $T$,
respectively, there exist open neighbourhoods $V(\chi)\subseteq U(\chi)$ and $V(\boldsymbol{0})\subseteq U(\boldsymbol{0})$ of $\chi$ and $\boldsymbol{0}$ in $T$, respectively, such that
\begin{equation*}
  V(\boldsymbol{0})\cdot V(\chi)\subseteq U(\boldsymbol{0}) \qquad \hbox{and} \qquad V(\chi)\cdot V(\boldsymbol{0})\subseteq U(\boldsymbol{0}).
\end{equation*}
By Theorem~\ref{theorem-3.1} every non-zero element of $\boldsymbol{B}_{\omega}^{\mathscr{F}_n}$ is an isolated point in $(\boldsymbol{B}_{\omega}^{\mathscr{F}_n},\tau)$ and by Corollary~3.3.11 of \cite{Engelking-1989} it is an isolated point of $T$, and hence the set $E(\boldsymbol{B}_{\omega}^{\mathscr{F}_n})\setminus V(\boldsymbol{0})$ is finite. Also Hausdorffness and compactness of $E(\boldsymbol{B}_{\omega}^{\mathscr{F}_n})$ imply that without loss of generality we may assume that $V(\chi)\cap E(\boldsymbol{B}_{\omega}^{\mathscr{F}_n})=\varnothing$. Since the neighbourhood $V(\chi)$ contains infinitely many elements of the semigroup $\boldsymbol{B}_{\omega}^{\mathscr{F}_n}$ and the set
$E(\boldsymbol{B}_{\omega}^{\mathscr{F}_n})\setminus V(\boldsymbol{0})$ is finite, there exists $(i,j,[0;k])\in V(\chi)$ such that either $(i,i,[0;k])\in V(\boldsymbol{0})$ or $(j,j,[0;k])\in V(\boldsymbol{0})$. Therefore, we have that at least one of the following conditions holds:
\begin{equation*}
  (V(\boldsymbol{0})\cdot V(\chi))\cap V(\chi)\neq \varnothing \qquad \hbox{and} \qquad (V(\chi)\cdot V(\boldsymbol{0}))\cap V(\chi)\neq \varnothing.
\end{equation*}
Every of the above conditions contradicts the assumption that $U(\chi)$ and $U(\boldsymbol{0})$ are disjoint open neighbourhoods of $\chi$ and $\boldsymbol{0}$ in $T$. The obtained contradiction implies the statement of the theorem.
\end{proof}

Since compactness preserves by continuous maps Theorems~\ref{theorem-2.13} and~\ref{theorem-3.8} imply

\begin{corollary}\label{corollary-3.9}
Let $n$ be a non-negative integer.
If $(\boldsymbol{B}_{\omega}^{\mathscr{F}_n},\tau)$ is a Hausdorff topological semigroup with the compact band then $(\boldsymbol{B}_{\omega}^{\mathscr{F}_n},\tau)$ is absolutely $H$-closed in the class of Hausdorff topological semigroups.
\end{corollary}

\begin{theorem}\label{theorem-3.10}
Let $n$ be a non-negative integer and $(\boldsymbol{B}_{\omega}^{\mathscr{F}_n},\tau)$ be a Hausdorff topological inverse semigroup.
If $(\boldsymbol{B}_{\omega}^{\mathscr{F}_n},\tau)$ is $H$-closed in the class of Hausdorff topological semigroups then its band $E(\boldsymbol{B}_{\omega}^{\mathscr{F}_n})$ is compact.
\end{theorem}

\begin{proof}
We shall prove the statement of the proposition for the semigroup $\mathscr{I}_\omega^{n+1}(\overrightarrow{\mathrm{conv}})$ which by Theorem~\ref{theorem-2.5} is isomorphic to the semigroup $\boldsymbol{B}_{\omega}^{\mathscr{F}_n}$.

Suppose to the contrary that there exists a Hausdorff topological inverse semigroup $(\mathscr{I}_\omega^{n+1}(\overrightarrow{\mathrm{conv}}),\tau)$ with the non-compact band such that $(\mathscr{I}_\omega^{n+1}(\overrightarrow{\mathrm{conv}}),\tau)$ is $H$-closed in the class of Hausdorff topological semigroups. By Theorem~\ref{theorem-3.1} every non-zero element of $\mathscr{I}_\omega^{n+1}(\overrightarrow{\mathrm{conv}})$ is an isolated point in $(\mathscr{I}_\omega^{n+1}(\overrightarrow{\mathrm{conv}}),\tau)$ and hence there exists an open neighbourhood $U(\boldsymbol{0})$ of the zero $\boldsymbol{0}$ in $(\mathscr{I}_\omega^{n+1}(\overrightarrow{\mathrm{conv}}),\tau)$ such that the set $A=E(\mathscr{I}_\omega^{n+1}(\overrightarrow{\mathrm{conv}}))\setminus U(\boldsymbol{0})$ is infinite and closed in $(\mathscr{I}_\omega^{n+1}(\overrightarrow{\mathrm{conv}}),\tau)$. Let $k$ be the smallest positive integer $\leqslant n+1$ such that the set $A_k=A\cap \mathscr{I}_\omega^{k}(\overrightarrow{\mathrm{conv}})$ is infinite for the subsemigroup $\mathscr{I}_\omega^{k}(\overrightarrow{\mathrm{conv}})$ of $\mathscr{I}_\omega^{n+1}(\overrightarrow{\mathrm{conv}})$. Without loss of generality we may assume that there exists an increasing sequence of non-negative integers $\left\{a_j\right\}_{j\in\omega}$ such that $a_0\geqslant n+1$ and
\begin{equation*}
  \widetilde{A}_k=\left\{
\left(
  \begin{smallmatrix}
    a_j & \cdots & a_j+k-1 \\
    a_j & \cdots & a_j+k-1 \\
  \end{smallmatrix}
\right)
\colon j\in\omega\right\}\subseteq A_k.
\end{equation*}
The continuity of the semigroup operation in $(\mathscr{I}_\omega^{n+1}(\overrightarrow{\mathrm{conv}}),\tau)$ implies that there exists an open neighbourhood $V(\boldsymbol{0})\subseteq U(\boldsymbol{0})$ of the zero $\boldsymbol{0}$ in $(\mathscr{I}_\omega^{n+1}(\overrightarrow{\mathrm{conv}}),\tau)$ such that $V(\boldsymbol{0})\cdot V(\boldsymbol{0})\subseteq U(\boldsymbol{0})$. By the definition of the semigroup operation on $\mathscr{I}_\omega^{n+1}(\overrightarrow{\mathrm{conv}})$ we have that the neighbourhood $V(\boldsymbol{0})$  does not contain at least one of the points
\begin{equation*}
  \left(
  \begin{smallmatrix}
    a_j & \cdots & a_j+n \\
    a_j & \cdots & a_j+n \\
  \end{smallmatrix}
\right)
\qquad \hbox{or} \qquad
\left(
  \begin{smallmatrix}
    a_j-n+k-2 & \cdots & a_j+k-1 \\
    a_j-n+k-2 & \cdots & a_j+k-1 \\
  \end{smallmatrix}
\right).
\end{equation*}
Since the both above points belong to $\mathscr{I}_\omega^{n+1}(\overrightarrow{\mathrm{conv}})\setminus \mathscr{I}_\omega^{n}(\overrightarrow{\mathrm{conv}})$, without loss of generality we may assume that there exists an increasing sequence of non-negative integers $\left\{b_j\right\}_{j\in\omega}$ such that $b_j+n+1<b_{j+1}$ for all $i\in\omega$ and
\begin{equation*}
  \widetilde{B}_{n+1}=\left\{
\left(
  \begin{smallmatrix}
    b_j & \cdots & b_j+n \\
    b_j & \cdots & b_j+n \\
  \end{smallmatrix}
\right)
\colon j\in\omega\right\}\nsubseteq V(\boldsymbol{0}).
\end{equation*}

Since $(\mathscr{I}_\omega^{n+1}(\overrightarrow{\mathrm{conv}}),\tau)$ is a Hausdorff topological inverse semigroup, the maps $\mathfrak{f}_1\colon \mathscr{I}_\omega^{n+1}(\overrightarrow{\mathrm{conv}})\to E(\mathscr{I}_\omega^{n+1}(\overrightarrow{\mathrm{conv}}))$, $\alpha\mapsto \alpha\alpha^{-1}$ and $\mathfrak{f}_2\colon \mathscr{I}_\omega^{n+1}(\overrightarrow{\mathrm{conv}})\to E(\mathscr{I}_\omega^{n+1}(\overrightarrow{\mathrm{conv}}))$, $\alpha\mapsto \alpha^{-1}\alpha$ are continuous, and hence the set $S_{\widetilde{B}_{n+1}}=\mathfrak{f}_1^{-1}(\widetilde{B}_{n+1})\cup \mathfrak{f}_2^{-1}(\widetilde{B}_{n+1})$ is infinite and open in $(\mathscr{I}_\omega^{n+1}(\overrightarrow{\mathrm{conv}}),\tau)$.

Let $\chi\notin \mathscr{I}_\omega^{n+1}(\overrightarrow{\mathrm{conv}})$. Put $S=\mathscr{I}_\omega^{n+1}(\overrightarrow{\mathrm{conv}})\cup\{\chi\}$. We extend the semigroup operation from $\mathscr{I}_\omega^{n+1}(\overrightarrow{\mathrm{conv}})$ onto $S$ in the following way:
\begin{equation*}
  \chi\cdot\chi=\chi\cdot\alpha=\alpha\cdot\chi=\boldsymbol{0}, \qquad \hbox{for all} \quad \alpha\in\mathscr{I}_\omega^{n+1}(\overrightarrow{\mathrm{conv}}).
\end{equation*}
Simple verifications show that such defined binary operation is associative.

For any $p\in\omega$ we denote
\begin{equation*}
  \Gamma_{p}=\left\{
\left(
  \begin{smallmatrix}
    b_{2j}   & \cdots & b_{2j}+n \\
    b_{2j+1} & \cdots & b_{2j+1}+n \\
  \end{smallmatrix}
\right)
\colon j\geqslant p\right\}.
\end{equation*}
We determine a topology $\tau_S$ on the semigroup $S$ in the following way:
\begin{enumerate}
  \item for every $\gamma\in \mathscr{I}_\omega^{n+1}(\overrightarrow{\mathrm{conv}})$ the bases of topologies $\tau$ and $\tau_S$ at $\gamma$ coincide; \quad and
  \item $\mathscr{B}(\chi)=\left\{U_p(\chi)=\{\chi\}\cup\Gamma_{p}\colon p\in \omega\right\}$ is the base of the topology $\tau_S$ at the point $\chi$.
\end{enumerate}
Simple verifications show that $\tau_S$ is a Hausdorff topology on the semigroup $\mathscr{I}_\omega^{n+1}(\overrightarrow{\mathrm{conv}})$.

For any $p\in \omega$ and any open neighbourhood $V(\boldsymbol{0})\subseteq U(\boldsymbol{0})$ of the zero $\boldsymbol{0}$ in $(\mathscr{I}_\omega^{n+1}(\overrightarrow{\mathrm{conv}}),\tau)$ we have that
\begin{equation*}
  V(\boldsymbol{0})\cdot U_p(\chi)=U_p(\chi)\cdot V(\boldsymbol{0})=U_p(\chi) \cdot U_p(\chi)=\{\boldsymbol{0}\}\subseteq V(\boldsymbol{0}).
\end{equation*}
We observe that the definition of the set $\Gamma_{p}$ implies that for any non-zero element
$\gamma=\left(
   \begin{smallmatrix}
     c & \cdots & c+l \\
     d & \cdots & d+l \\
   \end{smallmatrix}
 \right)
$  of the semigroup $\mathscr{I}_\omega^{n+1}(\overrightarrow{\mathrm{conv}})$ there exists the smallest positive integer $j_\gamma$ such that $c+l<b_{2j_\gamma}$ and $d+l<b_{2j_\gamma+1}$. Then we have that
\begin{equation*}
  \gamma\cdot U_{j_\gamma}(\chi)=U_{j_\gamma}(\chi)\cdot \gamma=\{\boldsymbol{0}\}\subseteq V(\boldsymbol{0}).
\end{equation*}
Therefore $(S,\tau_S)$ is a topological semigroup which contains $(\mathscr{I}_\omega^{n+1}(\overrightarrow{\mathrm{conv}}),\tau)$ as a dense proper subsemigroup. The obtained contradiction implies that $E(\boldsymbol{B}_{\omega}^{\mathscr{F}_n})$ is a compact subset of $(\mathscr{I}_\omega^{n+1}(\overrightarrow{\mathrm{conv}}),\tau)$.
\end{proof}

\begin{theorem}\label{theorem-3.11}
Let $n$ be a non-negative integer and $(\boldsymbol{B}_{\omega}^{\mathscr{F}_n},\tau)$ be a Hausdorff topological inverse semigroup. Then the following conditions are equivalent:
\begin{enumerate}
  \item\label{theorem-3.11(1)} $(\boldsymbol{B}_{\omega}^{\mathscr{F}_n},\tau)$ is $H$-closed in the class of Hausdorff topological semigroups;
  \item\label{theorem-3.11(2)} $(\boldsymbol{B}_{\omega}^{\mathscr{F}_n},\tau)$ is absolutely $H$-closed in the class of Hausdorff topological semigroups;
  \item\label{theorem-3.11(3)} the  band $E(\boldsymbol{B}_{\omega}^{\mathscr{F}_n})$ is compact.
\end{enumerate}
\end{theorem}

\begin{proof}
Implication \eqref{theorem-3.11(2)}$\Rightarrow$\eqref{theorem-3.11(1)} is obvious. Implications \eqref{theorem-3.11(1)}$\Rightarrow$\eqref{theorem-3.11(3)} and \eqref{theorem-3.11(3)}$\Rightarrow$\eqref{theorem-3.11(1)} follow from Theorem~\ref{theorem-3.10} and Theorem~\ref{theorem-3.8}, respectively.

Since a continuous image of a compact set is compact, Theorem~\ref{theorem-2.13} implies that \eqref{theorem-3.11(3)}$\Rightarrow$\eqref{theorem-3.11(2)}.
\end{proof}

The following example shows that a counterpart of the statement of Theorem~\ref{theorem-3.10} does not hold
when $(\boldsymbol{B}_{\omega}^{\mathscr{F}_n},\tau)$ be a Hausdorff topological semigroup.

\begin{example}\label{example-3.12}
On the semigroup $\mathscr{I}_\omega^{1}(\overrightarrow{\mathrm{conv}})$ we define a topology $\tau_\dagger$ in the following way. All non-zero elements of the semigroup $\mathscr{I}_\omega^{1}(\overrightarrow{\mathrm{conv}})$ are isolated points of $(\mathscr{I}_\omega^{1}(\overrightarrow{\mathrm{conv}}),\tau_{\dagger})$ and the family $\mathscr{B}_{\dagger}(\boldsymbol{0})=\left\{U_k(\boldsymbol{0})\colon k\in \omega \right\}$, where $U_k(\boldsymbol{0})=\left\{\boldsymbol{0}\right\}\cup\left\{\binom{2i}{2i+1}\colon i\geqslant k\right\}$, determines the base of the topology $\tau_{\dagger}$ at the point $\boldsymbol{0}$. It is obvious that $\tau_\dagger$ is a Hausdorff topology on $\mathscr{I}_\omega^{1}(\overrightarrow{\mathrm{conv}})$. Since $U_k(\boldsymbol{0})\cdot U_k(\boldsymbol{0})=\left\{\boldsymbol{0}\right\}$ for any $k\in\omega$ and $U_q(\boldsymbol{0})\cdot\big\{\binom{p}{q}\big\}=\big\{\binom{p}{q}\big\}\cdot U_p(\boldsymbol{0})=\left\{\boldsymbol{0}\right\}$ for any $p,q\in\omega$,
$(\mathscr{I}_\omega^{1}(\overrightarrow{\mathrm{conv}}),\tau_{\dagger})$ is a topological semigroup.
\end{example}

\begin{proposition}\label{proposition-3.13}
$(\mathscr{I}_\omega^{1}(\overrightarrow{\mathrm{conv}}),\tau_{\dagger})$ is $H$-closed in the class of Hausdorff topological semigroups.
\end{proposition}

\begin{proof}
Suppose to the contrary that there exists a Hausdorff topological semigroup $T$ which contains $(\mathscr{I}_\omega^{1}(\overrightarrow{\mathrm{conv}}),\tau_{\dagger})$ as a non-closed subsemigroup. Since the closure of a subsemigroup of a topological semigroup $S$ is a subsemigroup of $S$ (see \cite[p.~9]{Carruth-Hildebrant-Koch-1983}), without loss of generality we can assume that $\mathscr{I}_\omega^{1}(\overrightarrow{\mathrm{conv}})$ is a dense proper subsemigroup of $T$. Let $\chi\in T\setminus\mathscr{I}_\omega^{1}(\overrightarrow{\mathrm{conv}})$. Then $\boldsymbol{0}$ is the zero of the semigroup $T$ by Lemma 4.4 of \cite{Gutik=2018}, and $\chi\cdot\chi=\boldsymbol{0}$ by Proposition~\ref{proposition-3.6}.

Fix disjoint open neighbourhoods $U(\chi)$ and $U_p(\boldsymbol{0})$ of $\chi$ and $\boldsymbol{0}$ in $T$.
By Proposition~\ref{proposition-3.6}, $E(T)=E(\mathscr{I}_\omega^{1}(\overrightarrow{\mathrm{conv}}))$. By Theorem~1.5 of \cite{Carruth-Hildebrant-Koch-1983}, $E(\mathscr{I}_\omega^{1}(\overrightarrow{\mathrm{conv}}))$ is a closed subset of $T$ and hence  without loss of generality we can assume that $U(\chi)\cap E(\mathscr{I}_\omega^{1}(\overrightarrow{\mathrm{conv}}))=\varnothing$. Then for any open neighbourhoods $V(\chi)\subseteq U(\chi)$ and $U_q(\boldsymbol{0})\subseteq U_p(\boldsymbol{0})$ the infiniteness of $V(\chi)$ and the definition of the semigroup operation on $\mathscr{I}_\omega^{1}(\overrightarrow{\mathrm{conv}})$ that imply that
\begin{equation*}
V(\chi)\cdot U_q(\boldsymbol{0})\nsubseteq U_p(\boldsymbol{0}) \qquad \hbox{or} \qquad U_q(\boldsymbol{0})\cdot V(\chi)\nsubseteq U_p(\boldsymbol{0}),
\end{equation*}
which contradicts the continuity of the semigroup operation on $T$.
\end{proof}
\section*{Acknowledgements}

The authors acknowledge Serhii Bardyla and Alex Ravsky for their comments and suggestions.

\end{document}